\documentclass[sfdefaults=false, abstract=true]{scrartcl}
\usepackage[british]{babel}
\usepackage[utf8]{inputenc}
\usepackage{csquotes}
\usepackage[T1]{fontenc}
\usepackage{microtype}
\usepackage{amsmath}
\usepackage{amsthm}
\usepackage{amssymb}
\usepackage{mathtools}
\usepackage[
  backend=biber,
  style=alphabetic,
  sorting=nyt,
  natbib=true,
  maxnames=99,
  maxalphanames=5,
  hyperref=true,
  giveninits=true,
  doi=true,
  isbn=true,
  eprint=true,
  url=true
]{biblatex}
\usepackage{verbatim}
\usepackage{siunitx}  
\usepackage{xcolor}
\usepackage{dsfont}
\usepackage[normalem]{ulem}
\usepackage{needspace}
\usepackage{xspace}
\usepackage{enumitem,calc}
\usepackage{pdfpages}
\usepackage{subcaption}
\usepackage{hyperref}
\hypersetup{
	hidelinks
}

\theoremstyle{plain}
\newtheorem{theorem}{Theorem}[section]
\newtheorem{lemma}[theorem]{Lemma}
\newtheorem{proposition}[theorem]{Proposition}
\newtheorem{corollary}[theorem]{Corollary}

\theoremstyle{definition}
\newtheorem{definition}[theorem]{Definition}
\newtheorem{example}[theorem]{Example}
\newtheorem{remark}[theorem]{Remark}

\newtheorem*{standingassumptions}{Standing Assumptions}
\newtheorem*{datacon}{Data Conditions (D)}
\newtheorem*{hypocon}{Hypocoercivity Conditions (H)}

\newlist{assumptions}{description}{1}
\setlist[assumptions]{
    font=\normalfont,
  labelsep=0.6em,
  labelwidth=\widthof{($\Phi_2$7)},
  leftmargin=\dimexpr\labelwidth+\labelsep\relax,
  align=left,
  style=standard,
  itemindent=0pt,
  listparindent=0pt,
  parsep=0pt
}

\newlist{hypocoercivityassumptions}{description}{1}
\setlist[hypocoercivityassumptions]{
    font=\normalfont,
  labelsep=0.6em,
  labelwidth=\widthof{(H4)},
  leftmargin=\dimexpr\labelwidth+\labelsep\relax,
  align=left,
  style=standard,
  itemindent=0pt,
  listparindent=0pt,
  parsep=0pt
}

\newcommand{\defeq}{\mathrel{\vcentcolon=}}
\newcommand{\eqdef}{\mathrel{=\vcentcolon}}

\newcommand{\dx}{\textup{d}x}
\newcommand{\dy}{\textup{d}y}
\newcommand{\dmu}{\textup{d}\mu}
\newcommand{\dmui}[1]{\textup{d}\mu_{#1}}

\newcommand{\dxi}[1][i]{\partial_{x_{#1}}}
\newcommand{\dyi}[1][i]{\partial_{y_{#1}}}
\newcommand{\gradx}{\nabla_{x}}
\newcommand{\grady}{\nabla_{y}}

\newcommand{\ccinfty}{C_c^{\infty}}

\newcommand{\R}{\mathbb{R}}
\renewcommand{\P}{\mathbb{P}}

\DeclareMathOperator{\supp}{supp}

\DeclareMathOperator{\tr}{tr}
\DeclareMathOperator{\divergence}{div}

\newcommand{\SP}[3]{{\mathopen{(}{#1},{#2}\mathclose{)}}_{#3}}
\newcommand{\SPH}[2]{{\mathopen{(}{#1},{#2}\mathclose{)}}_{H}}
\newcommand{\ind}[1]{\mathds{1}_{#1}}
\newcommand{\Poincare}{Poincar\'e\xspace}

\newcommand{\konst}{c_2}
\newcommand{\Ctheta}{K_\theta}
\newcommand{\lmin}{\lambda_{\textup{min}}(QQ^*)}
\newcommand{\Cq}{\vert Q^*Q\vert_2}
\newcommand{\emax}{\varepsilon_\textup{max}(\theta)}
\newcommand{\sm}[1]{{\scriptscriptstyle#1}}
\newcommand{\Cell}{c_{\sm{\Sigma}}}
\newcommand{\Csigma}{M_{\sm{\Sigma}}}
\newcommand{\core}{\mathcal{D}}
\newcommand{\sL}{\mathcal{L}} 
\newcommand{\cL}{\mathcal{L}} 

\def\sB{\mathcal B}
\def\sE{\mathcal E}
\def\sF{\mathcal F}

\newcommand{\sZ}{\mathcal{Z}}

\def\wt{\widetilde}

\AtEveryBibitem{%
  \clearlist{location}%
  \clearfield{address}%
  \clearfield{issn}%
}
\AtEveryBibitem{%
  \ifentrytype{book}{\clearfield{pages}}{}%
  \ifentrytype{mvbook}{\clearfield{pages}}{}%
  \ifentrytype{collection}{\clearfield{pages}}{}%
  \ifentrytype{mvcollection}{\clearfield{pages}}{}%
  \ifentrytype{proceedings}{\clearfield{pages}}{}%
  \ifentrytype{mvproceedings}{\clearfield{pages}}{}%
}
\AtEveryBibitem{%
  \ifboolexpr{
    not test {\iffieldundef{doi}} or
    not test {\iffieldundef{eprint}} or
    not test {\iffieldundef{isbn}}
  }{%
    \clearfield{url}%
    \clearfield{urldate}%
  }{}%

  \ifentrytype{article}{%
    \clearfield{isbn}%
  }{}%
  
  \ifentrytype{book}{%
    \iffieldundef{isbn}{}{%
      \clearfield{doi}%
    }%
  }{}%
}

\numberwithin{equation}{section}
\areaset[current]{\textwidth}{\dimexpr\textheight+1.3\baselineskip\relax}

\title{Hypocoercivity for Hamiltonian Diffusions \\ with Singular Drift}

\author{Zhen-Qing Chen\textsuperscript{1} \and Martin Grothaus\textsuperscript{2} \and Onno Pfohl\textsuperscript{3,4,*}}

\date{}

\begin{document}
	\maketitle
    \thispagestyle{empty}
    \enlargethispage{4\baselineskip}
    \begingroup
    \renewcommand{\thefootnote}{\arabic{footnote}}
    \footnotetext[1]{Department of Mathematics, University of Washington, Seattle, WA 98195, USA. \\Email: \texttt{zqchen@uw.edu}.}
    \footnotetext[2]{Department of Mathematics, RPTU University Kaiserslautern–Landau, 67663 Kaiserslautern, Germany. \\Email: \texttt{grothaus@rptu.de}.}
    \footnotetext[3]{Institute of Mathematics, Technische Universität Berlin, 10623 Berlin, Germany. \\Email: \texttt{pfohl@math.tu-berlin.de}.}
    \footnotetext[4]{Department of Mathematics, Humboldt-Universität zu Berlin, 10099 Berlin, Germany.}
    \renewcommand{\thefootnote}{*}
    \footnotetext{Corresponding author.}
    \endgroup
	
	\begin{abstract}
        We establish $L^2$-exponential strong ergodicity (strong mixing) with an explicit rate of convergence for a class of degenerate diffusions with multiplicative noise and with singular drift in both the noisy and noise-free components. This class includes diffusions with an additional inert drift given by the gradient of a singular potential, as well as singular generalized stochastic Hamiltonian systems. Cases in which the diffusion is confined to a proper, bounded or unbounded subset of $\mathbb{R}^{d_1+d_2}$ are included. Concrete examples of admissible potentials are provided.
        
        To obtain these results, we use an analytical approach and study the long-time behavior of the strongly continuous contraction semigroup generated by the formal Kolmogorov backward operator. Using the theory of generalized Dirichlet forms, these objects are then identified with the transition semigroup and generator of the unique weak solution to the original stochastic differential equation. The existence and uniqueness of this solution are established under near-minimal conditions.  
     \end{abstract}


    \medskip
    
    \noindent \textbf{MSC Classification (2020)}: 37A25, 60H10, 37J25, 47D07
    
    \section{Introduction and Main Results}	
    \label{sec:intro}
    We study a class of degenerate stochastic dynamics on $\mathbb{R}^{d_1}\times\mathbb{R}^{d_2}$ in which the noise acts only in the second component, while both components may be driven by singular potentials.
    Concretely, we consider the stochastic differential equation (SDE)
    \begin{equation}
    	\label{eq:sde:general}
    	\begin{split}
    		\textup{d}X_t &= Q \nabla\phi_2(Y_t) \,\textup{d}t, \\
    		\textup{d}Y_t &= (\divergence \Sigma - \Sigma \nabla\phi_2)(Y_t) \,\textup{d}t - Q^*\nabla\phi_1(X_t) \,\textup{d}t + \sqrt{2} \sigma(Y_t) \,\textup{d}B_t,
    	\end{split}
    \end{equation}
    where $\sigma\colon \mathbb{R}^{d_2} \to \mathbb{R}^{d_2\times d_2}$ is a weakly differentiable matrix-valued function, $\Sigma = \sigma\sigma^*$ with $\sigma^*$ denoting the adjoint of $\sigma$, the potentials $\phi_i\colon \mathbb{R}^{d_i}\to \R \cup \{+\infty\}$ are possibly singular (that is, they may take the value $+\infty$) and continuous in the extended sense (that is, with respect to the usual topology on $(-\infty, \infty]$),
     $Q\in \mathbb{R}^{d_1\times d_2}$ is constant, and 
    \begin{equation*}
    	(\divergence \Sigma)_{i} (y) \defeq \sum_{j=1}^{d_2} \partial_{j} \Sigma_{ij}(y). 
    \end{equation*}
    Equation \eqref{eq:sde:general} describes a diffusion $(Y_t)_{t\geq 0}$ with an additional drift $- Q^*\nabla\phi_1(X_t)$. For $d_1=d_2$ and $Q^*\nabla\phi_1(X_t)=X_t$, it recovers diffusions with inert drift studied in Section~4 of \cite{BBCH10}. If $X_t$ and $Y_t$ are interpreted as position and velocity, respectively, then \eqref{eq:sde:general} generalizes the Langevin dynamics and is also known as a generalized stochastic Hamiltonian system. 
    
    Stochastic Hamiltonian systems are widely used in fields where physical quantities, such as energy, are perturbed by random influences or noise. They are primarily used to model the dynamics of complex systems whose exact microscopic state is unknown or whose flow is subject to external fluctuations. They are essential in physics for understanding the transition from microscopic chaos to macroscopic laws. In this context, the Hamilton equations are augmented by a friction term and a stochastic force, typically multiplicative noise. In the case of interacting particle systems, singular potentials are often used to model repulsion or confinement. That is why it is of great interest to examine both singular potentials and multiplicative noise.

    Hamiltonian systems with singular potentials in the deterministic component of the equation can be used to model stochastic differential equations with inert drift. In this model, the singular potential $\phi_2$ causes a confinement. The associated drift term blows up as the diffusion process approaches the boundary. For a discussion of such physical models, we refer the reader to \cite{MR1872429} and \cite{BURDZY2007278}.

    As the first main result of this paper, we establish existence and uniqueness of the weak solution to SDE \eqref{eq:sde:general} for each starting point in $\{\phi_1 <\infty \} \times \{\phi_2 <\infty \}$ under some nearly minimal conditions.

    \begin{theorem}
        \label{thm:existence_uniqueness_weak_sol}
        Assume that $\sigma$ is bounded and locally Lipschitz on $\R^{d_2}$ and $\Sigma :=\sigma \sigma^*$ is uniformly elliptic. Moreover, assume that $\phi_i$ is locally Lipschitz on $\Omega_i\defeq \{\phi_i < \infty\}$ for $i=1,2$, and that $Z_2:= \int_{\Omega_2} e^{- \phi_2 (y)} \,\dy <\infty$ and $\int_{\Omega_2} |\nabla \phi_2 (y)|^2 e^{-\phi_2 (y)} \,\dy<\infty$.
        Then, for every $(x, y)\in \Omega_1 \times \Omega_2$, SDE \eqref{eq:sde:general} has a unique weak solution $(X, Y)$ with $(X_0, Y_0)=(x, y)$ up to the lifetime $\zeta:= \inf\{t>0: (X_t, Y_t) \notin \Omega_1 \times \Omega_2\}$. 
	\end{theorem}

    As a consequence of Theorem \ref{thm:existence_uniqueness_weak_sol}, SDE \eqref{eq:sde:general} has unique weak solutions for arbitrary initial distributions under the same conditions, see Theorem \ref{T:2.4}. 

    The goal of this paper is to establish, under some suitable conditions, the $L^2$-exponential ergodicity of these solutions with an explicit convergence rate. 
    To this end, we study the transition semigroup of $(X, Y)$ and its generator using an analytic approach. We define a candidate invariant measure $\mu$ and a candidate generator $(\cL, D(\cL))$ of a strongly continuous contraction semigroup 
    on $L^2(\R^{d_1 + d_2},\mu)$.
    Our next main result, Theorem \ref{thm:main:abstract}, establishes the exponential $L^2$-convergence for this abstract semigroup. It is then identified with the transition semigroup of the weak solution to SDE \eqref{eq:sde:general} (for certain absolutely continuous initial distributions) in Theorem \ref{thm:identification}.
    
    Recall that $\Omega_i = \{x\in\mathbb{R}^{d_i} : \phi_i(x)<\infty\}$ for $i=1, 2.$ We set $\core \defeq C_c^2(\Omega_1\times\Omega_2)$. Then, by Itô's formula, every weak solution solves the martingale problem for the following formal Kolmogorov backward operator on $\core$:
    \begin{equation}
    	\label{eq:intro:generator}
    	\sL f \defeq \tr(\Sigma \nabla_y^2 f) + (\divergence \Sigma - \Sigma\grady\phi_2)\cdot \grady f
    	+ (Q\grady\phi_2)\cdot\gradx f - (Q^*\gradx\phi_1)\cdot\grady f.
    \end{equation}
    Here, the first $d_1$ coordinates of $E \defeq \mathbb{R}^{d_1}\times \mathbb{R}^{d_2}$ are denoted by $x$ and the last $d_2$ coordinates by $y$, the gradients $\gradx$, $\grady$, and the Hessian $\nabla_y^2$ are understood with respect to the corresponding component $x$ respectively $y$.

    The natural candidate for the equilibrium measure is the Gibbs measure 
    \begin{equation*}
    	\mu\defeq \mu_1\otimes\mu_2, \quad \textup{where}\quad \mu_i \defeq Z_i^{-1}e^{-\phi_i}\,\dx \quad \text{and}\quad
    	Z_i\defeq\int_{\mathbb{R}^{d_i}}e^{-\phi_i(x)}\,\dx<\infty.
    \end{equation*} 
    Indeed, under the assumptions below, integration by parts yields that $\sL^*\mu=0$ in the distributional sense, see Proposition \ref{prop:verification_data_conditions}.

    We regard $(\sL,\core)$ as an operator on $L^2(E,\mu)$. Note that the kernel of its symmetric part contains $L^2(\R^{d_1},\mu_1)$, so that its is not coercive on $\{ f\in L^2(E,\mu) : \mu(f)=0\}$. This makes studying the convergence of semigroups generated by extensions of $(\sL,\core)$ a problem of hypocoercive type, see \cite{Villani}.

    Under the assumptions below, $(\sL,\core)$ is densely defined and dissipative, hence closable. Its closure $(\cL,D(\cL))$ generates a strongly continuous contraction semigroup if and only if $(\sL, \core)$ is essentially $m$-dissipative, or equivalently $(\lambda I-\sL)(\core)$ is dense in $H$ for some 
    $\lambda >0$. In this case we can study the generated semigroup via the direct $L^2$-method proposed by Dolbeault, Mouhot, and Schmeiser in \cite{Dolbeault_Mouhot_Schmeiser_Hypocoercivity_article} and extended in \cite{P1_Grothaus_Stilgenbauer_HypocoercivityKolmogorovBackward}.

    Here, the condition of essential $m$-dissipativity has an equivalent probabilistic interpretation. Suppose that the transition function of a Markov solution to \eqref{eq:sde:general} leaves $\mu$ invariant and induces a strongly continuous contraction semigroup $(P_t)_{t\geq 0}$ on $L^2(E,\mu)$ with generator $(\cL_P,D(\cL_P))$. 
    Then, $(\sL, \core)$ is essentially $m$-dissipative if and only if $\core$ is a core of $(\cL_P,D(\cL_P))$.
    
    In \cite{P2_Grothaus_Stilgenbauer_LangevinRevisited} and \cite{Bertram_Grothaus_HypocoercivityLangevinMultipicativeNoise} the Hilbert space hypocoercivity method we employ was applied to the Langevin dynamics where $\phi_2=\tfrac{1}{2}\vert \cdot \vert^2$.
    As long as $\phi_2$ is assumed to be radially symmetric up to an affine coordinate transformation, see Condition ($\Phi_2 5$) below, many arguments generalize to our setting. 
    The key challenge is to establish auxiliary condition (H4) of the abstract Hilbert space hypocoercivity method discussed in Section \ref{sec:hypocoercivity_abstract}.
    
    In this context, previous work, e.g., \cite{P2_Grothaus_Stilgenbauer_LangevinRevisited}, \cite{P4_Grothaus_Wang_WeakPoincare_generalizedLangevinDynamics}, \cite{Bertram_Grothaus_HypocoercivityLangevinMultipicativeNoise} and \cite{Bertram_Grothaus_JDE_generalizedLangevinDynamics}, usually relied on the following restrictive growth condition on $\phi_1$:
    \begin{equation}
    	\label{eq:growth:old}
    	\lvert\nabla^2\phi_1(x)\rvert \leq C (1+\lvert \nabla \phi_1(x) \rvert) \qquad\text{for all }x\in\Omega_1 =\{\phi_1 <\infty\}\subset \mathbb{R}^{d_1},
    \end{equation}
    where $C\in(0,\infty)$ is constant.	But \eqref{eq:growth:old} does not allow singular potentials: if $\phi_1$ is continuous in the extended sense, lies in $C^2(\Omega_1)$ and satisfies \eqref{eq:growth:old}, then $\Omega_1=\emptyset$ or $\Omega_1 = \R^{d_1}$, see Proposition \ref{prop:growth}.
    Therefore, we assume instead that for all $\varepsilon>0$ there is a constant $C_\varepsilon\in (0,\infty)$ such that
    \begin{equation*}
		\vert\nabla^2\phi_1(x)\vert \leq \varepsilon\vert\nabla\phi_1(x)\vert^2 + C_\varepsilon \qquad\text{for all }x\in\Omega_1.  	
    \end{equation*}
	Under this condition, exponential convergence to equilibrium of the Langevin dynamics was obtained for singular potentials $\phi_1$ belonging to $C^\infty(\Omega_1)$ in \cite{P3_Camrud_LangevinSingularPotentials}.
	To this end, Camrud, Herzog, Stoltz, and Gordina derived new elliptic regularity estimates that we use in a slightly generalized version, see Theorem \ref{thm:apriori_estimates}.
	These estimates and an idea from \cite{Bertram_Grothaus_JDE_generalizedLangevinDynamics} allow us to establish the critical auxiliary condition (H4).

    There are complementary results on ergodicity with rate of convergence for singular degenerate stochastic differential equations. In the Langevin case with $C^\infty$-potentials, these were obtained in \cite{HerzogMattingly2019} via Lyapunov function techniques, in \cite{BaGoHe2021} via a combination of Gamma calculus and Lyapunov function techniques, and in \cite{GS15} via Dirichlet form and martingale techniques. The latter one works under the weakest assumption concerning smoothness of the potential in the position variable. For a general class of equations without the a priori assumption of the existence of an invariant measure, these were obtained in \cite{GroPanWang2024} using the Zvonkin transform to remove the singular drift. 
    
    However, the methods described above do not work in cases where both components have singular potentials. Hence, our approach is the first that allows a singular potential in both the noisy and the noise-free component. Concrete admissible potentials, not included in previous results, are presented in Example \ref{ex:examples} below.

    We now turn to our main results on the $L^2$-exponential strong ergodicity. 
    For convenience, we first list the assumptions used in the statements below. 
    We emphasize that these assumptions are not imposed globally, the relevant subset is specified explicitly in each section of the article.

    Verifying the essential $m$-dissipativity of $(\sL, \core)$ is not part of the hypocoercive arguments and is delicate in the singular setting considered here. We therefore present two versions of our main result. Theorem \ref{thm:main:abstract} proves exponential convergence in a general setting under the abstract assumption that $(\sL, \core)$ is essentially $m$-dissipative.
    Corollary \ref{thm:main:explicit} yields this convergence in a more restrictive setting, under assumptions on the potentials that can be checked explicitly.
    
    \medskip
	\noindent\textbf{Assumptions on the potentials.}
	\begin{assumptions}
        \item[($\Phi_i$1)]  The effective domain $\Omega_i =\{x\in\mathbb{R}^{d_i} : \phi_i(x)<\infty\} $ is open and nonempty. For all $k\in \mathbb{N}$, the sublevel set $\Omega_{i,k}\defeq\{x\in\mathbb{R}^{d_i} : \phi_i(x)<k\}$ has compact closure contained in $\Omega_i$. Moreover, $\phi_i\in C^2(\Omega_i)$ with $Z_i = \int_{\mathbb{R}^{d_i}} e^{-\phi_i(x)} \dx < \infty$.  
        
		\item[($\Phi_i$2)]   The \Poincare inequality holds for $\mu_i$, i.e., there is a constant $\Lambda_i >0$ such that
		\begin{displaymath}
			\bigl\Vert \nabla f\bigr\Vert_{L^2(\Omega_i,\mu_i)}^2 \geq \Lambda_i \bigl\lVert f -\textstyle\int\nolimits_{\Omega_i} f\,\dmui{i} \bigr\rVert_{L^2(\Omega_i,\mu_i)}^2 \qquad\textup{for all }f\in C_c^2(\Omega_i).
		\end{displaymath}
	
		\item[($\Phi_1$3)] It holds $\vert\nabla\phi_1\vert \in L^2(\Omega_1,\mu_1)$.
	
		\item[($\Phi_1$4)] For all $\varepsilon>0$ there is a constant $C_\varepsilon\in (0,\infty)$ such that
		\begin{equation*}
			\vert\nabla^2\phi_1(x)\vert \leq \varepsilon\vert\nabla\phi_1(x)\vert^2 + C_\varepsilon \qquad\text{for all }x\in\Omega_1=\{\phi_1 < \infty \}.
		\end{equation*}
		\item[($\Phi_2$3)] It holds $\vert\nabla\phi_2\vert \in L^4(\Omega_2,\mu_2)$ and $\vert\nabla^2\phi_2\vert \in L^2(\Omega_2,\mu_2)$.
		\item[($\Phi_2$4)] The potential $\phi_2$ is three times weakly differentiable on $\Omega_2$, and all third-order weak derivatives lie in 
        $L^2(\Omega_2, \mu_2)$. Furthermore, $\vert \nabla\phi_2 \vert \cdot \vert \nabla^2\phi_2 \vert \in  L^2(\Omega_2, \mu_2)$.
		\item[($\Phi_2$5)] There is a function $\psi:[0,\infty)\to\mathbb{R}\cup\{\infty\}$ with $\psi\in C^2(\{\psi<\infty\})$, an invertible matrix $\tau\in\mathbb{R}^{d_2\times d_2}$ and $b\in\mathbb{R}^{d_2}$ such that
		\begin{equation*}
			\phi_2(y) = \psi(\vert\tau y - b\vert^2) \qquad\text{for all } y\in\mathbb{R}^{d_2}.
		\end{equation*}
	\end{assumptions}
    \smallskip
    \noindent\textbf{Assumptions on the diffusion matrix.}
	\begin{assumptions}
		\item[($\Sigma$1)] $\Sigma$ is uniformly strictly elliptic on $\Omega_2$, i.e., there is a constant $\Cell \in (0,\infty)$ such that
		\begin{equation*}
			(v, \Sigma(y)v)_{\mathbb{R}^{d_2}} \geq \Cell \,\vert v\vert^2 \qquad\text{for all }v\in\mathbb{R}^{d_2}\text{ and } \mu_2\text{-almost all } y\in \Omega_2.
		\end{equation*}
		\item[($\Sigma$2)] $\Sigma$ is bounded and Lipschitz on $\Omega_2$, or equivalently $\Sigma_{ij} \in W^{1,\infty}(\Omega_2)$ for all $1\leq i,j \leq d_2$. We set 
		$\Csigma \defeq \max\{ \Vert \Sigma_{ij} \Vert_{L^{\infty}(\Omega_2)} , 
		\Vert \partial_k \Sigma_{ij} \Vert_{L^{\infty}(\Omega_2)} : 1\leq i,j,k \leq d_2 \}$.
	\end{assumptions}
    \smallskip
    \noindent\textbf{Functional analytic assumptions.}

    \noindent  For equivalent characterizations of essential $m$-dissipativity, we refer to Remark \ref{rem:ess-m-dissipativity}.
	\begin{assumptions}
		\item[(eL)] The operator $(\sL, \core)$ is essentially $m$-dissipative on $L^2(E,\mu)$.
		\item[(eT)] The operator $(T,C_c^2(\Omega_1))$ defined via
		$Tf \defeq \sum_{i,j=1}^{d_1}(Q\tau^* \tau Q^*)_{ij} (\partial_i\partial_j f-\partial_i\phi_1\partial_jf)$
		is essentially $m$-dissipative on $L^2(\mathbb{R}^{d_1},\mu_1)$.
	\end{assumptions}

	\begin{theorem}[General setting]
		\label{thm:main:abstract}
		Let $d_1 \leq d_2$, let $Q\in \mathbb{R}^{d_1\times d_2}$ be of full rank, and let $\Sigma\colon \mathbb{R}^{d_2} \to \mathbb{R}^{d_2\times d_2}$ be pointwise symmetric.
		Assume \textup{($\Phi_1$1)} - \textup{($\Phi_1$4)}, \textup{($\Phi_2$1)} - \textup{($\Phi_2$5)}, \textup{($\Sigma$1)}, \textup{($\Sigma$2)}, \textup{(eL)}, and \textup{(eT)}. 
	  Let $(T_t)_{t\geq 0}$ denote the strongly continuous contraction semigroup on $L^2(E,\mu)$ generated by the closure of $(\sL, \core)$.
        Then for each $C \in (1,\infty)$ there is $\lambda \in (0,\infty)$ such that
		\begin{equation}
			\label{eq:intro:main_con}
			\bigl\lVert T_t f - \textstyle\int\nolimits_{E} f \,\dmu \bigr\rVert_{L^2(E,\mu)} \leq C e^{-\lambda t} \bigl\lVert f - \textstyle\int\nolimits_{E} f \,\dmu \bigr\rVert_{L^2(E,\mu)} 
		\end{equation}
		for all $t\geq 0$ and $f\in L^2(E,\mu)$. Furthermore,
		\begin{equation*}
			\lambda = \frac{C-1}{C} \frac{\Cell}{n_1 + n_2 \Csigma + n_3 \Csigma^2}
		\end{equation*}
		for constants $n_1, n_2, n_3 \in (0,\infty)$ that are explicitly computable and independent of $\Sigma$.
	\end{theorem}

    \begin{remark}
		\label{rem:assumptions}
		\begin{enumerate}[label={\roman*)}]
			\item Explicit growth conditions for $\phi_i$ that imply ($\Phi_i$2) can be found in \cite{Poincare} and \cite{PoincareConvex}. While the \Poincare inequality is stated there for $\ccinfty(\Omega_i)$, it extends to $C_c^2(\Omega_i)$ under our assumptions since $\mu_i$ is finite.
			\item Condition (eT) is significantly less restrictive than (eL) and admits explicit sufficient criteria, see Lemma \ref{lem:essT:suff}. We retain the abstract formulation for readability.
			\item Assumption ($\Sigma$2) can be weakened at the expense of a more complicated expression for the convergence rate. Instead it is, e.g., sufficient to assume that $\Sigma$ is continuous and weakly differentiable on $\Omega_2$, that $\nabla\Sigma\in L_\textup{loc}^2(\Omega_2, \dx)$ and that 
			$\vert \Sigma\vert \cdot \vert \nabla\phi_2\vert^2$, $\vert \Sigma\vert \cdot \vert \nabla^2\phi_2\vert$,
			$\vert \nabla\Sigma\vert \cdot \vert \nabla\phi_2\vert$, $\vert \nabla\Sigma\vert \cdot \vert \nabla^2\phi_2\vert$,
			$\vert \Sigma\vert \cdot \vert \nabla\phi_2\vert \vert \nabla^2\phi_2\vert$, $\vert \Sigma\vert \cdot \vert \nabla^3\phi_2\vert$ lie in $L^2(\Omega_2,\mu_2)$.
			\item If $\phi_1$ or $\phi_2$ has no singularities, i.e., $\{\phi_1<\infty\}=\mathbb{R}^{d_1}$ or $\{\phi_2<\infty\}=\mathbb{R}^{d_2}$, a different choice of cut-off functions below (as in \cite{Bertram_Grothaus_HypocoercivityLangevinMultipicativeNoise}) may allow for weaker integrability assumptions.
		\end{enumerate}
	\end{remark}
    We next state the version with explicit, verifiable hypotheses.
    
    \medskip
    \Needspace{3\baselineskip}
    \noindent\textbf{Further assumptions on the potentials.}
    \begin{assumptions}
        \item[($\Phi_1$5)] Potential $\phi_1$ has no singularities, i.e., $\Omega_1 = \{\phi_1<\infty\} = \mathbb{R}^{d_1}$.
        \item[($\Phi_2$6)] There are constants $K\in (0,\infty)$ and $\alpha\in [1,2)$ such that
		\begin{displaymath}
			\vert \nabla^2\phi_2(y)\vert \leq K(1+\vert\nabla\phi_2(y)\vert^\alpha) \qquad\text{for all }y\in\Omega_2.
		\end{displaymath}
		\item[($\Phi_2$7)] It holds $\vert\nabla\phi_2\vert^2 e^{-\phi_2} \in L^\infty(\Omega_2,\dx)$.
    \end{assumptions}
    
	\begin{corollary}[Explicit conditions]
		\label{thm:main:explicit}
		Let $d\defeq d_1 = d_2 \in \mathbb{N}$ with $d\geq2$, let $Q \in \mathbb{R}^{d\times d}$ be invertible and $\Sigma \in \mathbb{R}^{d\times d}$ be symmetric and positive definite. Denote its smallest eigenvalue by $\lambda_\textup{min}(\Sigma)$.
		Assume \textup{($\Phi_1$1)} - \textup{($\Phi_1$5)}, \textup{($\Phi_2$1)} - \textup{($\Phi_2$7)}.  Then $(\sL, \core)$ is closable, and its closure generates a strongly continuous contraction semigroup $(T_t)_{t\geq 0}$ on $L^2(E,\mu)$.
		 For each $C \in (1,\infty)$ there is $\lambda \in (0,\infty)$ such that
		\begin{equation*}
			\bigl\lVert T_t f - \textstyle\int\nolimits_{E} f \,\dmu \bigr\rVert_{L^2(E,\mu)} \leq C e^{-\lambda t} \bigl\lVert f - \textstyle\int\nolimits_{E} f \,\dmu \bigr\rVert_{L^2(E,\mu)} 
		\end{equation*}
		for all $t\geq 0$ and $f\in L^2(E,\mu)$. Furthermore,
		\begin{equation*}
			\lambda = \frac{C-1}{C} \frac{\lambda_\textup{min}(\Sigma)}{n_1 + n_2 \vert \Sigma \vert + n_3 \vert \Sigma \vert^2}
		\end{equation*}
		for constants $n_1, n_2, n_3 \in (0,\infty)$ that are explicitly computable and independent of $\Sigma$. 
	\end{corollary}	
    For comments on the one-dimensional case $d=1$, we refer to Remark~\ref{rem:one dimensional case}.

	\begin{example}
		\label{ex:examples}
        \begin{enumerate}[label={\roman*)}]
        \item Singular potentials of the form $\phi_2(y)=(R-\vert \tau y-b\vert^2)^{-\gamma}$ for $\gamma>0$, $R>0$ and invertible $\tau$ satisfy \textup{($\Phi_2$1)} - \textup{($\Phi_2$7)}. For admissible potentials $\phi_1$  the component $(Y_t)_{t\geq 0}$ of the associated diffusion process (as constructed in Theorem \ref{thm:associated_process}) has paths in the bounded ellipsoid $\{\phi_2 < \infty \}$.
     
        \item  The potential $\phi_1=\exp(\frac{1}{2}\vert \cdot\vert^2)+ \vert \cdot \vert^3$ satisfies assumptions \textup{($\Phi_1$1)} - \textup{($\Phi_1$5)}, but $\phi_1$ violates the growth condition \eqref{eq:growth:old} and has only $C^2$ regularity. Therefore, it lies in a growth--regularity regime not covered by previous results with exponential convergence rates, in particular \cite[Theorem 1]{P3_Camrud_LangevinSingularPotentials} and \cite[Theorem 1.2]{Bertram_Grothaus_JDE_generalizedLangevinDynamics}. 
        \end{enumerate}
    \end{example}

   Under a subset of the assumptions of Theorem \ref{thm:main:abstract}, we obtain the following stochastic representation of the abstract semigroup $(T_t)_{t\geq 0}$ studied therein.

    \begin{theorem}
        \label{thm:identification}
        Let $\sigma\colon \R^{d_2} \to \R^{d_2\times d_2}$ be bounded and locally Lipschitz and let $\Sigma :=\sigma \sigma^*$ satisfy $(\Sigma1)$ and $(\Sigma2)$. For $i=1,2$, let the potential $\phi_i\colon \R^{d_i} \to\R\cup \{+\infty\}$ satisfy $(\Phi_i1)$ and $(\Phi_i 3)$. Assume $($eL$)$ and let $(T_t)_{t\geq 0}$ denote the strongly continuous contraction semigroup on $L^2(E,\mu)$ generated by the closure of $(\sL, \core)$. Let $h\in L^2(\mu)$ be a probability density with respect to $\mu$.
        Then the unique weak solution to SDE \eqref{eq:sde:general} with initial distribution $h\,\textup{d}\mu$ is a conservative diffusion process with state space $\Omega_1\times \Omega_2$, and
        its transition kernels $(p_t)_{t\geq 0}$ induce the operator semigroup $(T_t)_{t\geq 0}$ on $L^2(\Omega_1 \times \Omega_2,\mu) \cong L^2(E,\mu)$.
    \end{theorem} 
    
    This article is organized as follows. In Section~\ref{sec:solution_sde} we prove existence and uniqueness of the weak solution to \eqref{eq:sde:general} for every starting point $(X_0, Y_0)=(x, y) \in \Omega_1 \times \Omega_2$. In Section~\ref{sec:hypocoercivity_abstract}, the abstract hypocoercivity method as developed in \cite{Dolbeault_Mouhot_Schmeiser_Hypocoercivity_article} and \cite{P1_Grothaus_Stilgenbauer_HypocoercivityKolmogorovBackward} is formulated. In Section \ref{sec:hypocoercivity_application}, we apply it to our setting to prove Theorem \ref{thm:main:abstract}. In Section~\ref{sec:crit_ess-m-diss}, we present sufficient conditions for assumptions (eL) and (eT), which allow us to conclude Corollary \ref{thm:main:explicit}. In Section~\ref{sec:stochastic_interpretation},   
    we identify the semigroup $(T_t)_{t\geq 0}$ in  Theorem \ref{thm:main:abstract} with the transition semigroup of the diffusion process determined by the unique weak solutions of SDE \eqref{eq:sde:general} from Theorem    \ref{thm:existence_uniqueness_weak_sol}.
    
    \section{Existence and Uniqueness of Weak SDE Solutions}
    \label{sec:solution_sde}
	
	In this section, we show that SDE \eqref{eq:sde:general} has a unique weak solution for every starting point $(X_0, Y_0)=(x, y)\in \Omega_1 \times \Omega_2$ under a set of conditions that are nearly minimal. To keep the ideas transparent, we do not attempt to identify the minimal conditions. Clearly, the Lipschitz assumption can be weakened to a suitable weak differentiability condition.
    
    \begin{standingassumptions}
        Throughout this section, let $\sigma\colon \R^{d_2} \to \R^{d_2\times d_2}$ be bounded and locally Lipschitz and let $\Sigma :=\sigma \sigma^*$ be uniformly elliptic. For $i = 1,2$, let $\phi_i\colon \R^{d_i} \to \R\cup \{+\infty\}$ be continuous in the extended sense and locally Lipschitz on $\Omega_i =\{\phi_i < \infty\}$. Moreover, assume that $Z_2:= \int_{\Omega_2} e^{- \phi_2 (y)} \,\dy <\infty$ and $\int_{\Omega_2} |\nabla \phi_2 (y)|^2 e^{-\phi_2 (y)} \,\dy<\infty$.
    \end{standingassumptions}
  
    Note that these assumptions are implied by $(\Sigma1), (\Sigma2), (\Phi_11), (\Phi_21)$ and $(\Phi_23)$.
      
    \begin{proof}[Proof of Theorem \ref{thm:existence_uniqueness_weak_sol}]
        Let $W^{1,2}(\R^{d_2}):=\{ f\in L^2(\R^{d_2}, \dx): \nabla f \in L^2(\R^{d_2}, \dx) \}$ and define 
    	\[
    	\sE^0(u, v) := \int_{\R^{d_2}} \nabla u (x) \cdot \Sigma (x) \nabla v(x) \,\dx
    	\quad \textup{for } u, v\in W^{1,2}(\R^{d_2}).
    	\]
    	Since $ \Sigma :=\sigma \sigma^*$ is  bounded and uniformly elliptic, $(\sE^0, W^{1, 2}(\R^{d_2})) $ is a strongly local regular Dirichlet form on $L^2(\R^{d_2}, \dx)$. Its $L^2$-infinitesimal generator is the elliptic operator of divergence form $\sL^0 := \divergence (\Sigma \nabla)$ in the distributional sense. See \cite{CF12, FOT11} for these facts
    	and the terminologies. 
        By a celebrated result of Aronson \cite{Ar68}, the uniformly elliptic $\sL^0$ has a jointly H\"older continuous fundamental solution $p(t, x, y)$ on $(0, \infty)
    	\times \R^{d_2}\times \R^{d_2}$ that satisfies a two-sided Gaussian estimate and has the property $\int_{\R^{d_2}} p(t, x, y) \,\dy=1$ for every $t>0$ and $x\in \R^{d_2}$. It follows that there is a conservative Feller process $Y^0$ on $\R^{d_2}$ that has transition density function $p(t, x, y)$. 
        The Feller process $Y^0$ is symmetric with respect to the Lebesgue measure on $\R^{d_2}$, whose associated Dirichlet form on $L^2(\R^{d_2}, \dx)$  is $(\sE^0, W^{1, 2}(\R^{d_2}))$. 
        
    	Denote the law of the process $Y^0$  starting from $y$ by   $\P^0_y$. Since the coordinate function $x_i$ on $\R^{d_2}$	is locally in $W^{1,2}(\R^{d_2})$, we see from \cite[Theorems 5.5.1,  5.5.5 and Example 5.5.3]{FOT11} that there is a $d_2$-dimensional continuous martingale additive functional $B$ of $Y^0$ which is a Brownian motion on $\R^{d_2}$ under $\P^0_y$ for every $y\in \R^{d_2}$, and that under $\P^0_y$, 
    	\begin{equation} 
            \label{e:2.1}
            \textup{d}Y^0_t  = \divergence \Sigma  (Y^0_t) \,\textup{d}t + \sqrt{2} \sigma(Y^0_t) \,\textup{d}B_t \quad \textup{with } Y^0_0=y. 
    	\end{equation} 
      	Since $\sigma$ is bounded and locally Lipschitz and $\Sigma$ is uniformly elliptic on $\R^{d_2}$, 
    	it is well known via the Girsanov transform \cite[Theorems IV.2.2 and IV.4.2]{IW89} that for every $y\in \Omega_2$, there is a unique weak solution $\wt Y^0$ on $\R^{d_2}$ to SDE \eqref{e:2.1} up to the lifetime $\zeta^0:=  \lim_{k\to \infty} \inf\{t>0: \wt Y^0_t \notin B(0, k)\}$. 
    	Since $Y^0$ is a conservative weak solution to \eqref{e:2.1}, it follows that $\zeta^0 =\infty$ and $Y^0$ and $\wt  Y^0$ have the same distribution. Consequently, for every $y\in \R^{d_2}$, SDE \eqref{e:2.1} has a unique conservative weak solution that coincides with the Feller process $Y^0$ under $\P^0_y$.
    
    	Define  $\psi_2 (y) := Z_2^{-1/2} e^{-\phi_2 (y)/2}$ for $y\in \R^{d_2}$.
    	  Note that $\psi_2$ is continuous on $\R^{d_2}$ with  $\psi_2=0$ on $\Omega_2^c = \{ \phi_2= \infty\}$, and that $\psi_2 \in W^{1,2} (\Omega_2)$ as 
    	\begin{equation*}
            y \mapsto \nabla \psi_2 (y) =-  Z_2^{-1/2} e^{-\phi_2 (y)/2} \nabla \phi_2 (y) /2 \in L^2(\Omega_2, \dy).    
    	\end{equation*}
        In fact, $\psi_2 \in W^{1,2}_0(\Omega_2)\subset W^{1,2}(\R^{d_2})$. Indeed, for $n\geq 1$, define $v_n = (\psi_2-1/n)^+$.
        Then $v_n\in W^{1,2}(\Omega_2)$ with $\hbox{supp} [v_n] \subset \Omega_2 $ and so  $v_n\in W^{1,2}_0(\Omega_2)$. Since
        \begin{equation*}
            \lim_{n\to \infty} \left( \| \nabla (v_n - \psi_2) \|_{L^2( \Omega_2)}  +  \|  v_n - \psi_2 \|_{L^2( \Omega_2)} \right) =0,    
        \end{equation*}
        it follows that $\psi_2 \in W_0^{1,2} (\Omega_2)$.
        
        Define $M^0_t := - 2^{-1/2}\int_0^t  ( \sigma^* \nabla \phi_2 ) (  Y_s^0) \,\textup{d}B_s$ for $t<\tau^0_{\Omega_2}:= \inf\{s>0: Y^0_s \notin \Omega_2\}$,
        which is a continuous local martingale on the random time interval $[0, \tau^0_{\Omega_2})$. 
        More precisely, let $(D_k)_{k\geq 1}$ be a sequence of relatively compact open subsets of $\Omega_2$ that increases to $\Omega_2$ as $k\to \infty$. Then $\tau_k^0:=  \inf\{s>0: Y^0_s \notin D_k\}$ increases to $\tau^0_{\Omega_2}$ as $k\to \infty$ and, for each $k\geq 1$, $t\mapsto M^0_{t\wedge \tau^0_k}$ is a square integrable martingale. 
        
        Denote by $\mathrm{Exp} (M^0)$ the exponential local martingale of $M^0$ on $[0, \tau^0_{\Omega_2})$, and by $(\sF^0_t)_{t\geq 0}$ the minimal augmented filtration generated by the symmetric diffusion process $Y^0$.
        For each $y\in \Omega_2$, define a probability measure $\bar \P_y$ on 
        $\sF^0_{\tau^0_{\Omega_2}-}:=\sigma ( \cup_{k\geq 1} \sF^0_{\tau^0_k} )$ by 
        \[
            \frac{\textup{d} \bar \P_y}{\textup{d} \P^0_y} = \hbox{Exp} (M^0 )_{\tau^0_k} \quad \hbox{on } \sF^0_{\tau^0_k} 
        \]
        for each $k\geq 1$. 
        It is known as a special case of Theorem 2.6 of \cite{CFTYZ_2004} that the Girsanov transformed process $\bar Y:=( (Y^0_t)_{t\geq 0}, (\bar \P_y)_{y\in \Omega_2} )$ is a (conservative) recurrent $\mu_2$-symmetric diffusion process taking values in $\Omega_2$, whose Dirichlet form $(\sE, \sF)$ on $L^2(\Omega_2, \mu_2)$ is given by 
    	\begin{eqnarray*}
        	\sE (u, v) &:=& \int_{\Omega_2} \nabla u(y)  \cdot \Sigma (y) \nabla v(y) \,d\mu_2 (y) \quad \hbox{for } u, v\in \sF,
        	\  \hbox{ where } \\
        	\sF &:=& \hbox{the completion of } C^2_c (\Omega_2) \hbox{ with respect to } \sqrt{\sE_1},
    	\end{eqnarray*}
    	where $\dmu_2 (y)= \psi_2(y)^2 \,\dy$ and $\sE_1 (u, u):= \sE (u, u)+ \int_{\Omega_2} u(y)^2 \,\dmu_2 (y)$. 
    	For emphasis, we denote the Girsanov transform of $Y^0$ under $\bar \P_y$ by $\bar Y$.  
    	By the Girsanov theorem, $\bar Y$ satisfies 
    	\begin{equation} \label{e:2.2}
            \textup{d} \bar Y_t  = (\divergence \Sigma  - \Sigma \nabla\phi_2)(\bar Y_t) \,\textup{d}t + \sqrt{2} \sigma( \bar Y_t) \,\textup{d} \bar B_t
        \end{equation} 
        for some Brownian motion $\bar B$ on $\R^{d_2}$. Since SDE \eqref{e:2.1} has a unique weak solution, we have by \cite[Theorem IV.4.2]{IW89} and a localization argument that SDE \eqref{e:2.2} also has a unique weak solution.
        
        For $x\in \R^{d_1}$ and $t\geq 0$, define  
        \[
            X^x_t := x+ \int_0^t Q \nabla \phi_2 ( \bar Y_s) \,\textup{d} s.
        \]
       Since $\bar Y$ is conservative with continuous paths in $\Omega_2$, where $\nabla \phi_2$ is locally bounded, $X^x$ is well defined, continuous, and adapted to the minimal augmented filtration $(\bar{\sF}_t)_{t\geq 0}$ generated by $\bar Y$.
        Let $(U_k)_{k\geq 1}$ be a sequence of relatively compact subdomains of $\Omega_1$ that increases to $\Omega_1 = \{x\in \R^{d_1}: \phi _1 (x) < \infty\}$.   Then,  $T_k \defeq \inf\{ t>0: X^x_t \notin U_k\}$ is a $(\bar{\sF}_t)_{t\geq 0}$-stopping time.
    	Set $ \bar{\zeta} := \lim_{k\to \infty} T_k=\inf\{t> 0: X^x_t\notin \Omega_1\}$, and define
    	\[
            M^x_t\defeq  -2^{-1/2}  \int_0^t \sigma(\bar{Y}_s)^{-1} Q^* \nabla \phi_1 (X^x_s) \,\textup{d} \bar B_s \quad \hbox{for } t\in [0,  \bar{\zeta}).
        \]
        Since $\sigma^{-1}$ is bounded and $\nabla\phi_1$ is bounded on each $U_k$, $t\mapsto M^x_{t\wedge T_k}$ is a square integrable martingale under $\bar \P_y$ for every $k\geq 1$, and so $M^x$ is a continuous local martingale on $[0,  \bar\zeta)$ under $\bar \P_y$.
    	Again, denote by $\hbox{Exp} (M^x)$ the exponential local martingale of $M^x$ on $[0,  \bar\zeta )$. It induces a probability measure
    	$\P_{x, y}$ on $ \bar\sF_{\bar\zeta -}  :=\sigma (\cup_{k\geq 1}  \bar\sF_{T_k})$ by 
    	  \begin{equation} \label{e:2.3} 
    	   \frac{\textup{d}\P_{x, y}}{\textup{d}\bar \P_y}= \hbox{Exp} (M^x)_{T_k} \quad \hbox{on } \bar \sF_{T_k}
    	  \end{equation}
    	for each $k\geq 1$. By the Girsanov theorem, under $\P_{x, y}$, $\bar Y$ satisfies for some Brownian motion $W$:
    	  \[
          \textup{d} \bar Y_t = (\divergence \Sigma - \Sigma \nabla\phi_2)(\bar Y_t) \,\textup{d}t - Q^*\nabla\phi_1(X_t^{x}) \,\textup{d}t + \sqrt{2} 
    	   \sigma(\bar Y_t) \,\textup{d}W_t \quad \hbox{for  }  t\in [0, \bar\zeta )  
        \]
        with  $\bar Y_0=y.$
    	For emphasis, write $Y$ for $\bar Y$ under $\P_{x, y}$ and $X$ for $X^x$ under $\P_{x, y}$. Then $(X,Y)$ is a weak solution for the Langevin SDE \eqref{eq:sde:general} on $[0, \bar \zeta)$ starting in $(x,y)$.
    
        The weak solution to \eqref{eq:sde:general} is unique in law. Suppose $(\wt X, \wt Y)$ is another 
    	weak solution to \eqref{eq:sde:general} starting in $(x,y)$ whose underlying probability measure is denoted by $\wt \P_{x, y}$. Denote the corresponding Brownian motion by $\widetilde{B}$, let $(\widetilde{\sF}_t)_{t\geq 0}$ be the underlying filtration of this weak solution, and set $\widetilde{T}_k\defeq \inf \{ t>0 : \widetilde{X}_t \notin U_k \}$. Then we can do an \enquote{inverse} Girsanov transform to \eqref{e:2.3} by 
    	\[
            \frac{\textup{d} Q_{x, y}}{\textup{d} \wt \P_{x, y}} = \hbox{Exp} (\wt M)_{ \widetilde T_k} \quad \hbox{on } 
             \widetilde\sF_{\widetilde{T}_k}
        \]
    	for every $k\geq 1$, where $\wt M_t:= 2^{-1/2}  \int_0^t   \sigma^{-1}(\widetilde Y_s) Q^* \nabla \phi_1  (\wt X_s)\,\textup{d} \widetilde B_s$. By \cite[Theorem IV.4.2]{IW89}, under $Q_{x, y}$, $\wt Y$ is a weak solution to SDE \eqref{e:2.2}. By the weak uniqueness of \eqref{e:2.2}, we know that $\wt Y$ has the same distribution as $Y$. Moreover, $X$ and $\widetilde{X}$ are determined by $Y$ and $\widetilde{Y}$, respectively, through the first line of \eqref{eq:sde:general}. Consequently, after transforming back, $(\wt X, \wt Y)$ has the same distribution as $(X,Y)$.
    \end{proof}

	\begin{remark}  
        \begin{enumerate}[label={\roman*)}]
            \item The proof of Theorem \ref{thm:existence_uniqueness_weak_sol}  
             shows that the lifetime of the weak solution is determined by the $x$-component: as long as $X_t$ stays inside $\Omega_1$, $(X_t, Y_t)$ can continue moving inside $\Omega_1 \times \Omega_2$.
            \item If in addition \Poincare inequality condition ($\Phi_2$2) holds, the $\mu_2$-symmetric diffusion $\bar Y$ in the proof of Theorem \ref{thm:existence_uniqueness_weak_sol} is exponentially ergodic.
            \item The proof of Theorem \ref{thm:existence_uniqueness_weak_sol} does not require $Q$ to be constant. It yields unique weak solutions under the assumption that $Q$ is measurable and locally bounded on $\Omega_1 \times \Omega_2$.
        \end{enumerate}
	\end{remark}

    Collectively, by the weak existence and uniqueness of SDE \eqref{eq:sde:general}, the diffusion process $((X_t, Y_t)_{t\in [0, \zeta)}, (\P_{x, y})_{(x, y) \in \Omega_1\times \Omega_2})$ constructed above forms a continuous strong Markov process taking values in $\Omega_1 \times \Omega_2$ with possible explosion time $\zeta$; see, e.g.,  \cite[Theorem V.4.20]{KS91}. We now pass to the canonical path-space formulation.
    
    Denote by $(\Omega_1\times \Omega_2)_\partial := (\Omega_1\times \Omega_2 )\cup\{\partial\}$ the one-point compactification of $\Omega_1\times \Omega_2$, and define the path space 
    \[
        \sZ \defeq \{\omega \in C([0, \infty); (\Omega_1\times \Omega_2)_\partial ) : \omega(t)=\partial \textup{ for all }t\geq \zeta(\omega) \},
    \]
    where $\zeta(\omega)\defeq \inf \{t\geq 0 : \omega(t)=\partial \}$. 
    Set $\sB (\sZ)\defeq \sigma( \omega (t):t\geq 0)$. Let $\sF_t$ be the $\sigma$-field generated by $\{\omega (s): 0\leq s\leq t\}$ and $\sF_{t+}:=\cap_{s>t} \sF_s$.
    For each $(x, y) \in \Omega_1\times \Omega_2$, we denote by $\P_{(x, y)} $ the probability law induced on $\sZ$ by 
    the unique weak solution $(X_t, Y_t)_{t\in [0, \zeta)}$ of SDE \eqref{eq:sde:general} with $(X_0, Y_0)=(x, y)$.

    From now on, we use the same symbols $(X_t, Y_t)_{t\geq 0}$ for the coordinate process on $\sZ$, i.e. $(X_t,Y_t)(\omega)\defeq \omega(t)$, where the individual coordinates $X$ and $Y$ are only well-defined up to the lifetime.
    
	\begin{definition} 
        A probability measure $\P$ on $\sZ$ is said to be a solution to the martingale problem associated with $(\sL, \core)$ if for every $f\in \core=C^2_c(\Omega_1\times \Omega_2)$
        \[
            M^f_t:= f(X_t , Y_t ) - f(X_0, Y_0) - \int_0^t \sL f (X_s, Y_s) \,\textup{d}s , \quad t\geq 0, 
        \]
       where $f(\partial)\defeq 0$ and $(\sL f)(\partial)\defeq  0$, is a continuous martingale under $\P$ with respect to the filtration $(\sF_{t+})_{t\geq 0}$.
	\end{definition}
    
    By It\^o's formula, one concludes that $\P_{(x, y)}$ is a solution to the martingale problem associated with $(\sL, \core)$.
	In fact, by \cite[Corollary 5.4.9 and Proposition 5.4.11]{KS91} and their proofs, $\P_{(x, y)}$ is the unique
	solution to the martingale problem associated with $(\sL, \core)$  with initial condition $\P ((X_0, Y_0)=(x, y))=1$. 
	
	\begin{theorem} \label{T:2.4} 
	For any probability measure $\nu$ on $\Omega_1\times \Omega_2$, there is a unique solution to the martingale problem  associated with $(\sL, \core)$
	with initial distribution $\nu$. Moreover, SDE  \eqref{eq:sde:general}  has a unique weak solution $(X, Y)$ 
	with initial distribution $\nu$ up to the lifetime $\zeta:= \inf\{t>0: (X_t, Y_t) \notin \Omega_1 \times \Omega_2\}$.
	\end{theorem}
	
	\begin{proof}  As already noted above, for every $(x, y)\in \Omega_1\times \Omega_2$, the martingale problem
	associated with $(\sL, \core)$ has a unique solution $\P_{(x, y)}$ with initial value $(x, y)$. Thus by \cite[Theorem IV.5.1]{IW89},
	$(\P_{(x, y)})_{(x, y)\in \Omega_1\times \Omega_2}$ is a system of diffusion measures generated by $(\sL, \core)$
	in the sense of \cite[Definition IV.5.3]{IW89}.
    In particular, $(x, y) \mapsto \P_{(x, y)}(A)$ is Borel measurable for every $A\in \sB (\sZ)$. 
	For a probability measure $\nu$ on 	$\Omega_1\times \Omega_2$, 
	  define $\P_\nu:= \int_{\Omega_1\times \Omega_2} \P_{(x, y)} \,\textup{d}\nu(x,y)$. 
	  By  Fubini's theorem,  $\P_\nu$ is a solution to the martingale problem associated with
	  $(\sL, \core)$ with initial distribution $\nu$.  
      
    To establish uniqueness, suppose $\P$ is a solution to the martingale problem associated with $(\sL, \core)$ with initial distribution $\nu$. 
	Denote by $(\wt \P_{(x, y)})_{(x, y)\in \Omega_1 \times \Omega_2}$ the regular conditional probability
	of $\P$ given $(X_0, Y_0)$.  By \cite[Theorem 6.1.3]{SV06}, $\wt \P_{(x, y)}$ is a solution to the martingale problem	associated with $(\sL, \core)$  with initial value $(x, y)$ for  $\nu$-almost every $(x,y)\in \Omega_1 \times \Omega_2$. By uniqueness of solutions to this martingale problem, 
	$\wt \P_{(x, y)} = \P_{(x, y)}$ for $\nu$-almost every $(x, y)\in \Omega_1 \times \Omega_2$ and so $\P = \int_{\Omega_1\times \Omega_2} \wt \P_{(x, y)} \,\textup{d}\nu (x, y) =\P_\nu$. 
	This establishes existence and uniqueness of solutions to the martingale problem associated with $(\sL, \core)$ with initial distribution $\nu$.
	
	Now, by  Corollary 5.4.9 and Proposition 5.4.11 (and a localization argument in its proof) of \cite{KS91}, SDE \eqref{eq:sde:general} has a unique weak solution with initial distribution $\nu$ if and only if the martingale problem associated with $(\sL, \core)$ with initial distribution $\nu$ has a unique solution. This completes the proof of the theorem. 
	\end{proof} 
	
    Having established unique weak solutions to SDE \eqref{eq:sde:general}, we next study their long-time behavior. As outlined in the introduction, we use an analytic approach. 

	\section{The Hilbert Space Hypocoercivity Method}
	\label{sec:hypocoercivity_abstract}
	
	We recall the Hilbert space hypocoercivity method as formulated by Grothaus and Stilgenbauer in \cite[Section 2]{P1_Grothaus_Stilgenbauer_HypocoercivityKolmogorovBackward} and \cite[Section 2]{P2_Grothaus_Stilgenbauer_LangevinRevisited}. The framework is given by the following assumptions.
	\begin{datacon}
		\leavevmode\newline
		\vspace{-\baselineskip}
		\begin{itemize}[leftmargin = 1cm]
			\item[(D1)] Let $H$ denote the real Hilbert space $L^2(E,\mu)$ for some probability space $(E,\mathcal{F},\mu)$.
			\item[(D2)] Let $(T_t)_{t\geq 0}$ be a strongly continuous semigroup of bounded linear operators on $H$, and let the linear operator $(\cL,D(\cL))$ be its generator.
			\item[(D3)] Let $\core$ be a dense subspace of $H$ and an operator core of $(\cL,D(\cL))$.
			\item[(D4)] Let $(S,D(S))$ be a symmetric operator and $(A,D(A))$ be a closed antisymmetric operator on $H$ such that $\core \subset D(S)\cap D(A)$ and the decomposition $\cL=S-A$ holds on $\core$.
			\item[(D5)] Let $P$ be an orthogonal projection on $H$ such that $P(H) \subset D(S)$, $SP=0$ and $P(\core) \subset D(A^2)$. Let $P_S\colon H\to H$ be defined by $P_Sf= Pf +\SPH{f}{1}$ for $f\in H$.
		
			\item[(D6)] The measure $\mu$ is invariant for $(\sL,\core)$, that is, $\int_{E} \cL f \,\dmu = 0$ for all $f\in \core $.
			
			\item[(D7)] The semigroup $(T_t)_{t\geq 0}$ is conservative, or equivalently,  $1\in D(\cL)$ and $\cL 1=0$.
		\end{itemize}
	\end{datacon}

	In this setting, $A$ is closed, $P$ is bounded, and $\core \subset D(AP)$, so $(AP,D(AP))$ is closed and densely defined. Thus, there is a unique bounded linear operator $B$ on $H$ that extends $((I+(AP)^*(AP))^{-1}(AP)^*,D((AP)^*))$. This is due to the following well-known result from operator theory. For lack of a precise reference, we provide a proof in Appendix \ref{app:deferred_proofs}.

	\begin{proposition}
		\label{prop:vonNeumann}
		Let $(T,D(T))$ be a closed and densely defined operator on a Hilbert space $H$. Then $I+T^*T\colon D(T^*T)\to H$ is bijective and $((I+T^*T)^{-1}T^*, D(T^*))$ admits a unique bounded linear extension $B: H\to H$. It satisfies $\Vert B \Vert \leq 1$.
	\end{proposition}
   
	Within this framework, we can define the following hypocoercivity conditions.
	
	\begin{hypocon}
		\leavevmode\newline
		\vspace{-\baselineskip}
		\begin{hypocoercivityassumptions}
			\item[(H1)] It holds $PAP=0$ on $\core $.
			\item[(H2)] There is a constant $\Lambda_m>0$ such that
			\begin{equation}
				\label{eq:H2:abstract}
				-\SPH{Sf}{f} \geq \Lambda_m \Vert (I-P_S) f\Vert_H^2\qquad\text{for all }f\in \core .
			\end{equation}
			\item[(G)] The linear operator $(G,\core )\defeq (PA^2P,\core )$ is essentially self-adjoint on $H$.
			\item[(H3)] There is a constant $\Lambda_M>0$ such that
			\begin{equation}
				\label{eq:H3:abstract}
				\Vert APf \Vert_H^2 \geq \Lambda_M \Vert Pf\Vert_H^2\qquad\text{for all }f\in \core .
			\end{equation}
			\item[(H4)] There are constants $N_1, N_2 >0$ such that
            \begin{align}
                \label{eq:H4i:abstract}
				\Vert BS(I-P)f \Vert_H &\leq N_1 \Vert (I-P)f \Vert_H \qquad\text{for all }f\in \core,  \\
                \label{eq:H4ii:abstract}
		          \Vert BA(I-P)f \Vert_H &\leq N_2 \Vert (I-P)f \Vert_H\qquad\text{for all }f\in \core .
            \end{align}            
		\end{hypocoercivityassumptions}
	\end{hypocon}

    Under these conditions, there exists a modified entropy functional, equivalent to the $L^2$-distance to equilibrium, satisfying a coercive dissipation estimate. This leads to the following convergence theorem, see \cite[Theorem 2.18 and Corollary 2.13]{P1_Grothaus_Stilgenbauer_HypocoercivityKolmogorovBackward}. 
	\begin{theorem}
		\label{thm:hypocoercivity:abstract}
		If the data conditions \textup{(D)} and hypocoercivity conditions \textup{(H)} are satisfied, there are constants $\kappa_1,\kappa_2\in (0,\infty)$ such that
		\begin{equation}
			\label{eq:conv:abs}
			\Vert T_t f - \SPH{f}{1} \Vert_{H} \leq \kappa_1 e^{-\kappa_2 t} \Vert f - \SPH{f}{1} \Vert_{H} \qquad\text{for all }f\in H\text{ and }t\geq 0.
		\end{equation}
		Furthermore, $\kappa_1$ and $\kappa_2$ are explicitly computable in terms of $\Lambda_m$, $\Lambda_M$, $N_1$ and $N_2$. Specifically, if $\varepsilon, \delta \in (0,1)$ and $\kappa \in (0,\infty)$ satisfy
        \begin{equation}
            \label{eq:convrate:con1and2}
            \Lambda_m
            - \varepsilon \,(N_1+N_2+1)
              \bigl(1+\tfrac{1}{2\delta}\bigr)
            \geq \kappa
            \quad\textup{and}\quad
            \varepsilon\,
            \bigl(
              \tfrac{\Lambda_M}{1+\Lambda_M}
              - \tfrac{\delta}{2}(N_1+N_2+1)
            \bigr)
            \geq \kappa,
        \end{equation}
		then \eqref{eq:conv:abs} holds for $\kappa_1 = (\frac{1+\varepsilon}{1-\varepsilon})^{1/2}$ and $\kappa_2 = \frac{\kappa}{1+\varepsilon}$.
	\end{theorem}
	
	To verify (H4), we make use of \cite[Lemma 2.3]{Bertram_Grothaus_HypocoercivityLangevinMultipicativeNoise}, which states the following.
	\begin{lemma}
		\label{lem:H4:suffcrit}
		Let $(T,D(T))$ be a linear operator on $H$ satisfying $\core \subset D(T)$ and $AP(\core )\subset D(T^*)$. If there is a constant $N<\infty$ such that 
		\begin{equation*}
			\lVert T^*APf\rVert_H \leq N \lVert (I-G)f\rVert_H \qquad\text{for all } f\in \core ,
		\end{equation*}
		and \textup{(G)} holds, then $(BT, D(T))$ is bounded by $N$.
	\end{lemma}

   Regarding conditions (D2) and (G) and the related assumptions (eL) and (eT), we recall the following from the theory of operator semigroups, see e.g.\ \cite{Pazy, Goldstein}. 
   
    \begin{remark}
    \label{rem:ess-m-dissipativity}
        If an operator $(R,D(R))$ on $H$ is densely defined and dissipative, i.e., $(Rf, f)_H \leq 0$ for all $f\in D(R)$, 
                then the following are equivalent: 
       \begin{enumerate}        
            \item [(i)]    $(R,D(R))$ is essentially $m$-dissipative;
            \item [(ii)]   its closure generates a strongly continuous operator semigroup (which then is contractive); 
            \item [(iii)]   $(\lambda I-R)(D(R))$ is dense in $H$ for some (and hence all) $\lambda >0$.
        \end{enumerate}
        If, moreover, $(R,D(R))$ is symmetric,
        it is essentially $m$-dissipative if and only if it is essentially self-adjoint.
    \end{remark}
		
	\section{Application to Hamiltonian Diffusions with Singular Potentials in Both Variables}
	\label{sec:hypocoercivity_application}
	
	To prove our main result, we show that Theorem \ref{thm:hypocoercivity:abstract} can be applied in our setting of Hamiltonian diffusions with singular potentials in both variables: In Subsection \ref{subsec:data_conditions}, we supplement the definition from the introduction to establish a framework as described by the data conditions (D). In Subsection \ref{subsec:hypocoercivity_conditions}, we verify the hypocoercivity conditions (H). Finally, in Subsection \ref{subsec:main_result}, we establish the main result and compute the convergence rate.
	
	\subsection{The Data Conditions}
	\label{subsec:data_conditions}

    First, we fix a set of assumptions and recall the setting from the introduction.

    \begin{standingassumptions}
        Throughout this subsection, let $\Sigma\colon \mathbb{R}^{d_2} \to \mathbb{R}^{d_2\times d_2}$ be pointwise symmetric and satisfy ($\Sigma2$). Let $\phi_1$ and $\phi_2$ be $\R\cup \{+\infty\}$-valued potentials on $\mathbb{R}^{d_i}$ satisfying ($\Phi_i$1) and ($\Phi_i$3), respectively.
    \end{standingassumptions}
    
	\begin{definition}
		\label{def:recall_intro}        
		Let $\mu_i$ denote the probability measure on $\mathbb{R}^{d_i}$ obtained by normalization of $e^{-\phi_i}\dx$, let $\Omega_{i}$ denote the set $\{ \phi_i < \infty \}$, and $\Omega =\Omega_1 \times \Omega_2$.
		Let $\core =C_c^2(\Omega)$ and $H = L^2(E,\mu)$ with $\mu=\mu_1\otimes\mu_2$ and $E = \mathbb{R}^{d_1}\times \mathbb{R}^{d_2}$.
		Let $(\sL, \core)$ on $H$ be defined by \eqref{eq:intro:generator}.
	\end{definition} 
	
	\begin{remark}
		\label{rem:conventions}
        The density of $\mu$ with respect to $\dx$ is strictly positive on $\Omega$ and vanishes on $\Omega^c$. Thus $H \cong L^2(\Omega,\mu)$, and every $f\in L^2(\Omega,\mu)$ admits at most one continuous representative on $\Omega$. 
        We can therefore identify $\mu$-square-integrable elements of $C(\Omega)$ with their natural embedding into $H$.
		
		The spaces $L^2(\mathbb{R}^{d_1},\mu_1)$ and $L^2(\mathbb{R}^{d_2},\mu_2)$ embed canonically and isometrically into $H$ by viewing elements as  \enquote{functions on $\R^{d_1}\times\R^{d_2}$ constant w.r.t. $y$ and $x$, respectively}. Thus we can view  $L^2(\mu_i)\defeq L^2(\mathbb{R}^{d_i},\mu_i)$ as a closed subspace of $H$ and may use the simplified notation $\Vert \cdot\Vert$ for both the norm on $H$ and on $L^2(\mu_i)$ without risk of confusion.
	\end{remark}
	
	Next, we introduce the projections $P$ and $P_S$. The following characterization is a consequence of Fubini's theorem.
	\begin{lemma}
		\label{lem:IntroProjections}
        Let $P_S$ and $P$ denote the orthogonal projections of $H$ onto $L^2(\mu_1)$ and $\{f\in L^2(\mu_1) : \SPH{f}{1}=0\}$ respectively. Then for $f\in H$ it holds
        \begin{equation*}
            P_S f = \int_{\mathbb{R}^{d_2}} f\,\dmui{2} \defeq 
			\int_{\mathbb{R}^{d_2}}f(\cdot,y)\,\dmui{2}(y) \qquad\text{and} \qquad Pf = P_S f-\SPH{f}{1}.
        \end{equation*}
	\end{lemma}
     Here, $\int_{\mathbb{R}^{d_2}} f\,\dmui{2}$ denotes the equivalence class containing all measurable extensions of $x\mapsto \int_{\mathbb{R}^{d_2}} f(x,\cdot)\,\dmui{2}$, which is defined $\mu_1$-almost everywhere for any version of $f \in H$.
    
    Using dominated convergence to differentiate under the integral sign and to ensure continuity of the derivative, see e.g. \cite[Theorem 11.4 and 11.5]{Schilling}, we obtain the following regularity result, which we use without further reference.

	\begin{lemma}
		\label{lem:PS:diff}
		For $f\in C_c^k(\Omega)$ and $k\in\mathbb{N}$ we have that $P_Sf \in C_c^k(\Omega_1)$ and $\partial_i P_S f = P_S \dxi f$. 
	\end{lemma}
	Here and in the following, we write $\partial_i g$ instead of $\dxi g$ or $\dyi g$ for a function $g$ depending only on one of the coordinates $x$ or $y$.
	
	\begin{lemma}
		\label{lem:D_dense}
		$\ccinfty(\Omega)$ is dense in $H$ and $\ccinfty(\Omega_i)$ is dense in $L^2(\mu_i)$ for $i=1,2$.
	\end{lemma}
	\begin{proof}
		As $\mu$ restricted to $\Omega$ is a finite Borel measure on an open subset of $\mathbb{R}^{d_1 + d_2}$, it is regular and $C_c(\Omega)$ dense in $L^2(\Omega,\mu)$, see e.g. \cite[Proposition 7.2.3 and 7.4.2]{Cohn}. But every $C_c(\Omega)$-function is the uniform limit (and thus $L^2(\Omega, \mu)$-limit) of $\ccinfty(\Omega)$-functions. The same arguments apply to $\ccinfty(\Omega_i) \subset L^2(\mu_i)$.
	\end{proof}

	Next, we introduce the symmetric and antisymmetric part of $(\sL, \core)$.

	\begin{definition}
		\label{def:SAL}
		Define $(S,\core )$ and $(A,C_c^1(\Omega))$ via
		\begin{align*}
			Sf &\defeq \tr(\Sigma \nabla_y^2f) + (\divergence \Sigma - \Sigma \nabla\phi_2)\cdot\grady f, \\
			Af &\defeq (Q^*\nabla\phi_1)\cdot\grady f - (Q\nabla\phi_2)\cdot\gradx f.
		\end{align*}
		Then $\sL=S-A$ on $\core $ and $(S,\core )$, $(A,C_c^1(\Omega))$ and $(\sL, \core )$ are well-defined linear operators on $H$: Since $\divergence \Sigma$ is bounded by a multiple of $\Csigma$, all terms in the definition of $Sf$ and $Af$ are (bounded by) functions from $C_c(\Omega)\subset H$.
	 
	\end{definition}	
	Integration by parts yields the following identities.
	\begin{lemma}
		\label{lem:bilinearforms}
		Let $g\in C^1(\Omega)\cap H$. Then
		\begin{equation*}
			\SPH{Sf}{g} = - \int_{E} (\Sigma\grady f)\cdot \grady g \,\dmu \qquad\text{for }f\in \core, 
		\end{equation*}
		and
		\begin{equation*}
			\SPH{Af}{g} = \int_{E} (Q\grady f)\cdot \gradx g - (Q\grady g)\cdot \gradx f \,\dmu \qquad\text{for }f\in C_c^1(\Omega).
		\end{equation*}
	\end{lemma}
	By Lemmas \ref{lem:D_dense} and \ref{lem:bilinearforms}, the operators $(S,\core )$, $(A,C_c^1(\Omega))$ and $(\sL, \core )$ are densely defined and dissipative on $H$. As such they admit closures $(S,D(S))$, $(A,D(A))$ and $(\cL ,D(\cL))$, see e.g. \cite[Theorem 4.5]{Pazy}.
	Before we investigate these operators, we introduce cutoff functions as in \cite[Section 5]{P3_Camrud_LangevinSingularPotentials}.
	\begin{definition}
		\label{def:cutoffs}
        For $n\in\mathbb{N}$ let $\chi_n \in C^\infty(\mathbb{R})$ be such that $\ind{(-\infty,n]} \leq \chi_n \leq \ind{(-\infty,n+1)}$ and both $\vert \chi_n'\vert$ and $\vert \chi_n''\vert$ are bounded by a constant $C(\chi) \in (0,\infty)$ independent of $n$.
	\end{definition}
	
	\begin{lemma}
		\label{lem:cutoffs}
		For $i=1,2$ and $n\in\mathbb{N}$ we have that $\chi_n(\phi_i)$, $\chi_n'(\phi_i)$, $\chi_n''(\phi_i)\in C_c^2(\Omega_i)$ and
		\begin{equation*}
			\chi_n(\phi_i) \to 1 \quad\text{and}\quad \chi_n'(\phi_i),\chi_n''(\phi_i) \to 0
		\end{equation*}
		pointwise on $\Omega_i$ as $n\to\infty$ and uniformly bounded by 1 respectively $C(\chi)$.
	\end{lemma}
	
	\begin{proof}
        Since $\Omega_{i,n}= \{\phi_i<n\}$, the identity $\ind{(-\infty,n]} \leq \chi_n \leq \ind{(-\infty,n+1)}$ implies that
        \begin{equation*}
			\ind{\Omega_{i,n}} \leq \chi_n(\phi_i) \leq \ind{\Omega_{i,n+1}}\qquad\text{and}\qquad \vert \chi_n'(\phi_i) \vert, \vert \chi_n''(\phi_i) \vert \leq C(\chi) \ind{\Omega_{i,n+2}\setminus \Omega_{i,n}}.
		\end{equation*}
        The claim follows since the sets $\Omega_{i,n}$, $n\in\mathbb{N}$, form a sequence exhausting $\Omega_{i}$ and have compact closure contained in $\Omega_i$ by assumption ($\Phi_i$1).
	\end{proof}

    Using these cutoff functions for approximation, dominated convergence, and the integrability assumptions ($\Phi_1$3) and ($\Phi_2$3), we can verify some natural formulas for $(S,D(S))$, $(A,D(A))$, and $(\cL ,D(\cL))$. The proof is deferred to the appendix.
    
	\begin{lemma}
		\label{lem:D5D7}
		\begin{enumerate}[label={\roman*)}]
			\item $L^2(\mu_1)\subset D(S)$ and $S|_{L^2(\mu_1)}=0$. In particular, $P(H)\subset D(S)$ and $SP=0$.
			\item $1\in D(A)$ and $A1 = 0$. Moreover, $P(\core )\subset D(A)$ and, for $f\in \core$,
			\begin{equation*}
				AP f = AP_S f = -( Q\nabla\phi_2)\cdot\gradx  P_S f.
			\end{equation*}
			\item $AP(\core)\subset D(A)$ and, for $f\in \core$,
			\begin{equation}
				\label{eq:AAP}
				A^2P f = ( Q \nabla\phi_2)\cdot ((\gradx^2 P_S f) Q\nabla\phi_2) - ( Q^*\nabla\phi_1)\cdot ((\nabla^2\phi_2) Q^* \gradx P_S f).
			\end{equation}
			\item  $1\in D(\cL )$ and $\cL 1=0$.
		\end{enumerate}
	\end{lemma}

    To conclude this section, we note that all data conditions (D) hold if (eL) is satisfied.

    \begin{proposition}[Verification of (D)]
        \label{prop:verification_data_conditions}
        If \textup{(eL)} holds, the data conditions \textup{(D)} hold.
    \end{proposition}

    \begin{proof}
        Since $(\cL,D(\cL))$ is the closure of $(\sL, \core)$, (D2) is equivalent to (eL), and (D3) follows from Lemma \ref{lem:D_dense}. Conditions \textup{(D1)} and -- up to (anti-)symmetry -- \textup{(D4)} hold by construction. (Anti-)symmetry of $(S,\core)$ and $(A,C_c^1(\Omega))$ follows from Lemma \ref{lem:bilinearforms} and extends to the closures. (D5) and (D7) follow from Lemmas \ref{lem:IntroProjections} and \ref{lem:D5D7}. Finally, $\SPH{\cL f}{1}=\SPH{Sf}{1}-\SPH{Af}{1}=0$ for all $f\in \core$ by Lemma \ref{lem:bilinearforms}, which proves (D6). 
    \end{proof}
	
	\subsection{The Hypocoercivity Conditions}
	\label{subsec:hypocoercivity_conditions}
    Next, we verify the hypocoercivity conditions (H) under the assumptions of our main result, Theorem \ref{thm:main:abstract}. To simplify the calculations, we additionally assume that $\phi_2$ is radially symmetric. This assumption can be dropped in the proof of Theorem \ref{thm:main:abstract} by means of a suitable coordinate transformation,   which is detailed in Appendix \ref{app:details_coordinate_trafos}.

     \begin{standingassumptions}
        Throughout this subsection, let $d_1 \leq d_2$, let $Q\in \mathbb{R}^{d_1\times d_2}$ be of full rank, and let $\Sigma\colon \mathbb{R}^{d_2} \to \mathbb{R}^{d_2\times d_2}$ be pointwise symmetric satisfying \textup{($\Sigma$1)} and \textup{($\Sigma$2)}. Let \textup{($\Phi_1$1)} - \textup{($\Phi_1$4)}, \textup{($\Phi_2$1)} - \textup{($\Phi_2$5)}, \textup{(eL)} and \textup{(eT)} hold. Furthermore, let $\tau=I$ and $b=0$ in ($\Phi_2$5), so that $\phi_2$ is radially symmetric.
    \end{standingassumptions}
	
	\subsubsection{Hypocoercivity Assumptions (H1) - (H3)}

    The symmetry of $\phi_2$ results in the following integral identities, whose proof can be found in the appendix.
    
	\begin{lemma}
		\label{lem:phi2_integrals}
		For $i,j\in\{1,\ldots,d_2\}$, it holds that $\int_{\Omega_2} \partial_i\phi_2 \,\dmui{2} = 0$ and
		\begin{displaymath}
			\int_{\Omega_2} \partial_i\partial_j\phi_2 \;\dmui{2} = \int_{\Omega_2} \partial_i\phi_2 \,\partial_j\phi_2 \;\dmui{2} =  \konst \,\delta_{ij}, \quad\text{where}\quad  \konst \defeq \frac{1}{d_2} \Vert \nabla\phi_2\Vert_{L^2(\mu_2)}^2.
		\end{displaymath}
	\end{lemma}

	\begin{lemma}[Verification of (H1)]
		\label{lem:H1}
		The identity $PAP=0$ holds on $\core$.
	\end{lemma}
	\begin{proof}
		Let $f\in \core$. Then $APf = -\nabla\phi_2 \cdot (Q^*\gradx P_S f) $ by Lemma \ref{lem:D5D7}, and hence
		\begin{displaymath}
			P_S A P f = \int_{\Omega_2} APf \,\dmui{2}
			= -\sum_{i=1}^{d_2}  (Q^*\gradx P_S f)_i \int_{\Omega_2} \partial_i\phi_2 \,\dmui{2} = 0
		\end{displaymath}
		due to Lemma \ref{lem:phi2_integrals}. Using Fubini's theorem to rewrite $\SPH{APf}{1}$, we conclude that
		\begin{displaymath}
			PAP f = P_S APf - \SPH{APf}{1} = P_S APf - \SP{P_S APf}{1}{L^2(\mu_1)} = 0. \qedhere
		\end{displaymath}
	\end{proof}
	
	\begin{lemma}[Verification of (H2)]
		\label{lem:H2}
		For all $f\in \core$, we have the estimate
		\begin{equation}
			\label{eq:Lambda_m}
			-(Sf,f)_H \geq \Lambda_m \Vert (I-P_S) f\Vert_H^2 \qquad\text{with}\quad \Lambda_m = \Cell \Lambda_2.
		\end{equation} 
	\end{lemma}
	\begin{proof}
		Let $f\in \core$. By Lemma \ref{lem:bilinearforms}, ($\Sigma$1) and Fubini's theorem we have
		\begin{multline*}
			-\SPH{Sf}{f} = \int_{\Omega} \SP{\grady f}{\Sigma \grady f}{\mathbb{R}^{d_2}} \,\dmu 
			\geq \Cell \int_{\Omega} \lvert \grady f \rvert^2 \,\dmu \\
			= \Cell \int_{\Omega_1} \Vert \grady f(x,\cdot)\Vert_{L^2(\Omega_2,\mu_2)}^2  \,\dmui{1}(x).
		\end{multline*}
		Note that $f(x,\cdot)\in C_c^2(\Omega_2)$ for all $x\in\Omega_1$. Hence the \Poincare inequality from ($\Phi_2$2) implies
		\begin{equation*}
			-(Sf,f)_H \geq \Cell \Lambda_2 \int_{\Omega_1} \bigl\lVert f(x,\cdot) - \textstyle\int\nolimits_{\Omega_2} f(x,y)\,\dmui{2}(y) \bigr\rVert_{L^2(\mu_2)}^2  \,\dmui{1}(x),
		\end{equation*}
		where the integral on the right hand side equals $\Vert f - P_S f \Vert_{H}^2$ by Fubini's theorem.
	\end{proof}
	
	\begin{lemma}[Verification of (H3)]
		\label{lem:H3}
		For all $f\in \core$, we have the estimate
		\begin{equation}
			\label{eq:Lambda_M}
			\Vert APf \Vert_H^2 \geq \Lambda_M \,\Vert Pf\Vert_H^2 \qquad\text{with}\quad\Lambda_M = \Lambda_1\konst\, \lmin, 
		\end{equation}
		where $\lmin>0$ denotes the smallest eigenvalue of $QQ^*$.
	\end{lemma}
	Note that $QQ^*$ is invertible since, by assumption, $Q\in \mathbb{R}^{d_1\times d_2}$ has full rank and $d_1 \leq d_2$. So its smallest eigenvalue is indeed positive.
	\begin{proof}
		Let $f\in \core$. By Lemma \ref{lem:D5D7} we have
		\begin{equation*}
			\begin{split}
				\bigl\lVert APf \bigr\rVert_H^2
				& = \bigl\lVert \nabla\phi_2\cdot( Q^*\gradx P_S f) \bigr\rVert_H^2
				  = \int_{\Omega_1} \int_{\Omega_2} \Bigl(\sum_{i=1}^{d_2} \partial_i\phi_2 \,( Q^*\gradx P_S f)_i\Bigr)^2\,\dmui{2}\,\dmui{1} \\
				& = \sum_{i,j=1}^{d_2} \;\int_{\Omega_2} \partial_i\phi_2 \,\partial_j\phi_2 \;\dmui{2} \;\;\int_{\Omega_1} ( Q^*\gradx P_S f)_i ( Q^*\gradx P_S f)_j \,\dmui{1}.
			\end{split}
		\end{equation*}
		In view of Lemma \ref{lem:phi2_integrals} and the definition of $\lmin$, this yields
		\begin{equation*}
			\Vert APf \Vert_H^2
			= \konst \int_{\Omega_1} \vert  Q^*\gradx P_S f\vert^2 \,\dmui{1}
			\geq \konst \,\lmin \Vert \gradx P_S f\Vert_{L^2(\mu_1)}^2.
		\end{equation*}
		Applying the \Poincare inequality for $\mu_1$, see ($\Phi_1$2), to $P_S f \in C_c^2(\Omega_1)$, we obtain that
		\begin{equation*}
			\Vert \gradx P_S f\Vert_{L^2(\mu_1)}^2 
			\geq \Lambda_1 \Vert P_S f - \SP{P_S f}{1}{L^2(\mu_1)} \Vert_{L^2(\mu_1)}^2.
		\end{equation*}
		Finally,  $\SP{P_S f}{1}{L^2(\mu_1)} = \SP{f}{1}{H}$ by Fubini's theorem, and therefore
		\begin{equation*}
			\Vert P_S f - \SP{P_S f}{1}{L^2(\mu_1)} \Vert_{L^2(\mu_1)}^2 = \Vert P f \Vert_{L^2(\mu_1)}^2 = \Vert P f \Vert_H^2. \qedhere
		\end{equation*}
	\end{proof}
	
	\subsubsection{Hypocoercivity Assumption (G)}
    \label{subsubsec:G}

    Before we prove that $(G,\core)$ is essentially self-adjoint, we note that it is well-defined by (D5) and satisfies $\SPH{Gf}{g} = -\SPH{APf}{APg}$ for all $f,g\in \core$ since $A$ is antisymmetric and $P$ is an orthogonal projection. In particular, $(G,\core)$ is symmetric and negative semidefinite.

    Moreover, $G$ can be expressed in terms of the operator $(T,C_c^2(\Omega_1))$ from assumption (eT), 
    which, since $\tau= I$, is given by
	\begin{equation*}
        Tf =
        \sum_{i,j=1}^{d_1}( Q Q^*)_{ij} (\partial_i\partial_jf-\partial_i\phi_1\partial_jf).
	\end{equation*}
    For $f\in C_c^2(\Omega_1)$ and $g\in C^1(\Omega_1)\cap L^2(\mu_1)$, integration by parts yields
    \begin{equation}
        \label{eq:T}
        \SP{Tf}{g}{L^2(\mu_1)} = -  \int_{\Omega_1} ( Q^*\nabla f)\cdot ( Q^*\nabla g)\,\dmui{1}.
	\end{equation}
    In particular, $T$ is a densely defined (see Lemma \ref{lem:D_dense}), symmetric and negative semidefinite operator on $L^2(\mu_1)$.
	
	\begin{lemma}
		\label{lem:GT}
		For all $f\in \core$, we have that
		\begin{equation}
			\label{eq:GT}
			Gf = PA^2Pf = \konst T P_S f, \qquad\text{where}\quad\konst = \tfrac{1}{d_2} \Vert\nabla\phi_2\Vert_{L^2(\mu_2)}^2.
		\end{equation}
	\end{lemma}
	\begin{proof}
		Let $f\in \core$ and recall from Lemma \ref{lem:D5D7} that 
		\begin{displaymath}
			A^2P f = ( Q \nabla\phi_2)\cdot (\gradx^2( P_S f) \, Q\nabla\phi_2) - ( Q^*\nabla\phi_1)\cdot (\nabla^2\phi_2\,  Q^* \gradx P_S f).
		\end{displaymath}
        To determine the image of $A^2Pf$ under $P_S$, we integrate both terms on the right-hand side against $\mu_2$. Using Lemma \ref{lem:phi2_integrals} and invariance of the trace under cyclic permutations, we compute that
        \begin{multline*}
				\int_{\Omega_2} ( Q \nabla\phi_2)\cdot (\gradx^2( P_S f) \, Q\nabla\phi_2) \,\dmui{2}
				= \sum_{i,j} (Q^*\gradx^2( P_S f) \, Q)_{ij} \int_{\Omega_2} \partial_i\phi_2 \partial_j\phi_2 \,\dmui{2} \\
				= \sum_{i,j} (Q^*\gradx^2( P_S f) \, Q)_{ij} \cdot \konst \,\delta_{ij}
				= \konst \tr\bigl(Q^*\gradx^2( P_S f)Q\bigr) 
				= \konst \tr\bigl(QQ^*\gradx^2( P_S f)\bigr),
		\end{multline*}
        and
        \begin{multline*}
			\int_{\Omega_2} (Q^*\nabla\phi_1)\cdot (\nabla^2\phi_2\, Q^* \gradx P_S f) \,\dmui{2} \\
			= (Q^*\nabla\phi_1)\cdot \Bigl( \int_{\Omega_2}\nabla^2\phi_2 \,\dmui{2} \Bigr) (Q^* \gradx P_S f)
			= \konst (Q^*\nabla\phi_1)\cdot (Q^* \gradx P_S f).
		\end{multline*}
		Combining both results, we obtain
		\begin{equation*}
			P_S A^2Pf
			= \konst \sum_{i,j=1}^{d_1}(QQ^*)_{ij} (\partial_{i}\partial_{j} P_S f-\partial_i\phi_1\partial_j P_S f)
			= \konst T P_S f.
		\end{equation*}
		Using Fubini's theorem and \eqref{eq:T}, we observe that
		\begin{displaymath}
			\SPH{A^2Pf}{1} = \SP{P_S A^2Pf}{1}{L^2(\mu_1)} = \konst \SP{T P_S f}{1}{L^2(\mu_1)}=0,
		\end{displaymath}
		and thus
		\begin{equation*}
			Gf = PA^2Pf = P_SA^2Pf + \SPH{A^2Pf}{1} =  \konst T P_S f. \qedhere
		\end{equation*}		
	\end{proof}
	
	\begin{theorem}[Verification of (G)]
		\label{thm:G_selfadjoint}
		$(G,\core)$ is essentially self-adjoint.
	\end{theorem}
	\begin{proof}
		By Remark \ref{rem:ess-m-dissipativity}, the claim is equivalent to $(I-G)(\core)$ being dense in $H$. 
		Therefore, we let $g\in H$ satisfy
		\begin{equation}
			\label{eq:thm:G_SA:1}
			\SPH{(I-G)f}{g}=0 \qquad\text{for all }f\in \core,
		\end{equation}
		and show that $g=0$. In view of the Hahn-Banach and the Riesz representation theorem, this proves the claim.

        Let $f\in C_c^2(\Omega_1)$. We define $f_n\defeq \chi_n(\phi_2)f\in \core$ and $\omega_n\defeq \int_{\mathbb{R}^{d_2}} \chi_n(\phi_2) \,\dmui{2}$ for $n\in\mathbb{N}$. Then $ P_S f_n =  \omega_n f$ and by dominated convergence $f_n\to f$ in $H$ and $\omega_n \to 1$ in $\mathbb{R}$. Hence
		\begin{equation*}
			(I-G) f_n = (I- \konst T P_S) f_n
			= f_n -  \omega_n \konst Tf \to (I-\konst T)f \qquad\text{as}\quad n\to\infty.
		\end{equation*}
		Since $f$ and $Tf$ are independent of $y$, using Fubini's theorem and \eqref{eq:thm:G_SA:1}, we obtain that
		\begin{equation}
			\label{eq:thm:G_SA:2}
            \SP{(I - \konst T)f}{P_S g}{L^2(\mu_1)} = \SPH{(I- \konst T)f}{g} = \lim_{n\to\infty} \SPH{(I-G)f_n}{g} = 0.
		\end{equation}
		By Remark \ref{rem:ess-m-dissipativity}, assumption (eT) implies that $(\konst^{-1} I - T)(C_c^2(\Omega_1))$ is dense in $L^2(\mu_1)$. It thus follows from \eqref{eq:thm:G_SA:2} that $P_S g = 0$.
		
		Since $Gf$ is independent of $y$, we conclude that
		$\SPH{Gf}{g} = \SP{Gf}{P_Sg}{L^2(\mu_1)} = 0$ for all $f\in \core$.
		Combining this with \eqref{eq:thm:G_SA:1},
		\begin{displaymath}
			\SPH{f}{g} = \SPH{(I-G)f}{g} + \SPH{Gf}{g} = 0 
		\end{displaymath}
		for all $f\in \core$, and $g=0$ follows as $\core$ is dense in $H$.		
	\end{proof}
	
	\subsubsection{Hypocoercivity Assumption (H4) - First Inequality}	
	
	To verify the first inequality of (H4), we make use of an auxiliary operator as used in
	\cite[Section 4, in part. Lemma 4.3]{Bertram_Grothaus_JDE_generalizedLangevinDynamics}.
	
	\begin{lemma}
		\label{lem:R}
		For $k=1,\ldots,d_2$, define
		\begin{equation*}
			\omega_k \defeq \sum_{i,j=1}^{d_2} \bigl( \partial_i \Sigma_{ij} \, \partial_j\partial_k \phi_2 
			+ \Sigma_{ij} (\partial_i\partial_j\partial_k \phi_2 
			-  \partial_i \phi_2 \, \partial_j\partial_k\phi_2)\bigr) \in L^2(\mu_2).
		\end{equation*}
		Then
		\begin{equation*}
			Rf \defeq \omega\cdot (Q^*\gradx P_S f) \in H
            \qquad\text{and}\qquad \Vert Rf \Vert_H \leq \Csigma \, C(\phi_2) \,\Vert Q^* \gradx P_S f\Vert_{L^2(\mu_1)}
		\end{equation*}
		for all $f\in \core$, where $C(\phi_2)\in (0,\infty)$ is a constant depending only on $\phi_2$.
	\end{lemma}
	
	\begin{proof}
		By ($\Sigma2$) we have
		\begin{equation*}
			\vert \omega_k \vert \leq \Csigma \sum_{i,j=1}^{d_2} \bigl(\vert \partial_j\partial_k \phi_2 \vert + \vert \partial_i\partial_j\partial_k \phi_2 \vert + \vert \partial_i \phi_2 \vert \, \vert \partial_j\partial_k\phi_2 \vert \bigr)
		\end{equation*}
		almost everywhere. So the integrability assumptions ($\Phi_2 3$) and ($\Phi_2 4$) yield that $\omega_k \in L^2(\mu_2)$ for all $k$ and that $\Vert \omega \Vert_{L^2(\mu_2)} \leq \Csigma C(\phi_2)$ for a constant $C(\phi_2)\in (0,\infty)$.
		
		Now, for $f\in \core$ the Cauchy-Schwarz-inequality and Fubini's theorem imply that
		\begin{equation*}
			\Vert Rf \Vert^2_H 
			\leq \int_{\Omega} \lvert \omega\rvert^2 \lvert Q^*\gradx P_S f\rvert^2 \,\dmu 
			= \Vert \omega \Vert_{L^2(\mu_2)}^2 \Vert Q^* \gradx P_S f\Vert_{L^2(\mu_1)}^2,
		\end{equation*}
		where $\vert Q^* \gradx P_S f \vert \in C_c(\Omega_1)\subset L^2(\mu_1)$.
	\end{proof}
	
	\begin{theorem}[Verification of (H4), 1st inequality]
		\label{thm:H4:IE1}
        For all $f\in \core$, it holds that
		\begin{equation}
			\label{eq:N1}
			\Vert BS(I-P)f \Vert_H \leq N_1 \Vert (I-P)f \Vert_H \quad\text{with}\quad N_1 = \Csigma C(\phi_2)/\sqrt{2\konst}.
		\end{equation}
	\end{theorem}
	
	\begin{proof}
		Due to Lemma \ref{lem:H4:suffcrit},
        it suffices to show for all $f\in \core$ that
		\begin{equation}
			\label{eq:H4:IE1:intermediate}
			APf \in D(S^*), \quad S^*APf = Rf, \quad\text{and}\quad \Vert Rf\Vert_H \leq N_1 \Vert (I-G)f\Vert_H.
		\end{equation}
		Let $f, g\in \core$. Then $APf= -\nabla\phi_2\cdot Q^*\gradx P_S f \in C^1(\Omega)$ by Lemma \ref{lem:D5D7}. Thus Lemma \ref{lem:bilinearforms} and integration by parts yield   
		\begin{align}
				\SPH{Sg}{APf}
				&= \sum_{i,j=1}^{d_2} \int_{\Omega} (\dyi g) \, \Sigma_{ij} \, (\dyi[j] APf) \,\dmu \notag \\
				&= \sum_{i,j,k=1}^{d_2} \int_{\Omega} (\dyi g) \, \Sigma_{ij} \, (-\partial_j\partial_k \phi_2) (Q^*\gradx P_Sf)_k \,\dmu \notag \\
				&= \sum_{k=1}^{d_2} \int_{\Omega} g \, \omega_k \, (Q^*\gradx P_Sf)_k \,\dmu
				= \SPH{g}{Rf}. \label{eq:S*AP}
		\end{align}
		Since $\core$ is a core of $(S,D(S))$, \eqref{eq:S*AP} extends to all $g\in D(S)$. This proves the first half of \eqref{eq:H4:IE1:intermediate}. It remains to estimate $\Vert Rf\Vert_H$.
		
        We note that
	    $\Vert (I-\konst T) g\Vert^2_{L^2(\mu_1)} = \Vert g\Vert^2_{L^2(\mu_1)} - 2\konst \SP{Tg}{g}{L^2(\mu_1)} + \Vert \konst Tg \Vert^2_{L^2(\mu_1)}$
        for $g\in C_c^2(\Omega_1)$, and use \eqref{eq:T} to obtain
		\begin{equation*}
			2\konst  \Vert Q^* \gradx P_S f\Vert^2_{L^2(\mu_1)} = -2\konst \SP{TP_Sf}{P_Sf}{L^2(\mu_1)} \leq \Vert (I-\konst T)P_S f\Vert^2_{L^2(\mu_1)}.
		\end{equation*}
		We conclude from Lemma \ref{lem:R} that
		\begin{equation*}
			\Vert Rf \Vert_H \leq \Csigma C(\phi_2) \Vert Q^* \gradx P_S f\Vert_{L^2(\mu_1)} \leq N_1 \Vert (I-\konst  T)P_S f\Vert_{L^2(\mu_1)}.
		\end{equation*}
		By \eqref{eq:GT} and since $P_S=I$ on $P(H)\subset L^2(\mu_1)$, we have  $\konst TP_Sf = Gf = P_SGf$  and thus
		\begin{equation}
            \label{eq:I-kTvsI-G}
			\Vert (I-\konst T)P_S f\Vert_{L^2(\mu_1)} = \Vert P_S (I-G)f \Vert_{L^2(\mu_1)} \leq \Vert (I-G)f \Vert_H,
		\end{equation}
        which concludes the proof.
	\end{proof}
	
	\subsubsection{Hypocoercivity Assumption (H4) - Second Inequality}
	To verify the second inequality in (H4), we make use of elliptic a priori estimates from \cite{P3_Camrud_LangevinSingularPotentials} in an adapted, slightly generalized form.
	
	Let $\theta\in (0,1)$. Then ($\Phi_1$4) implies that there is a positive constant $\Ctheta \in (0,\infty)$ such that 
	\begin{equation*}
		\label{eq:def:Ctheta}
		\vert \Delta\phi_1\vert \leq \theta \vert \nabla\phi_1\vert^2 + \Ctheta\qquad \text{on } \Omega_1.
	\end{equation*}
	Set
	\begin{equation*}
		\label{eq:def:epsmax}
		\emax \defeq \frac{\lmin}{4 \Cq}  (1-\theta)^2,
	\end{equation*}
    where $\vert \cdot\vert_2$ denotes the spectral norm, and introduce for all $0<\varepsilon < \emax$ the constants
	\begin{equation*}
		\label{eq:def:xi_12}
		\begin{aligned}
			\xi_{\varepsilon, \theta}^{(1)} &\defeq \frac{1}{1-\frac{\varepsilon}{\emax}} \, \max\biggl(\frac{1}{\konst^2},\frac{\Cq}{2\konst}\Bigl(\varepsilon \frac{2\Ctheta}{1-\theta} + C_\varepsilon\Bigr)\biggr), \\
			\xi_{\varepsilon, \theta}^{(2)} &\defeq  \frac{1}{1-\frac{\varepsilon}{\emax}} \, \max\biggl(\frac{1}{\konst^2}\frac{\Cq^{-1}}{\emax},\frac{1}{2\konst}\Bigl(\frac{2\Ctheta}{1-\theta} + \frac{C_\varepsilon}{\emax}\Bigr)\biggr).
		\end{aligned}
	\end{equation*}
	Note that $\xi_{\varepsilon, \theta}^{(1)} > \konst^{-2}$ and $\xi_{\varepsilon, \theta}^{(2)} > 4 \konst^{-2} \lmin ^{-1}$. Hence the infima
	\begin{equation*}
		\label{eq:def:xi}
		\xi_1 \defeq \inf \xi_{\varepsilon, \theta}^{(1)} \qquad\text{and}\qquad \xi_2 \defeq \inf \xi_{\varepsilon, \theta}^{(2)}
	\end{equation*}
	with respect to all $\theta \in (0,1)$ and $0<\varepsilon< \emax$ define positive constants.
	
	\begin{theorem}
		\label{thm:apriori_estimates}
		For all $f\in C_c^2(\Omega_1)$, we have the estimates
        \begin{align}
            \lVert  Q^* \nabla^2 f \, Q \rVert_{L^2(\mu_1)}^2 
			&\leq \xi_1 \lVert (I-\konst T)f \rVert_{L^2(\mu_1)}^2, \label{eq:aprioriestimate:1} \\
            \bigl\lVert \lvert\nabla\phi_1\rvert \,\lvert  Q^* \nabla f \rvert \bigr\rVert_{L^2(\mu_1)}^2 
			&\leq \xi_2 \lVert (I-\konst T)f \rVert_{L^2(\mu_1)}^2. \label{eq:aprioriestimate:2}
        \end{align}
	\end{theorem}
	Before proving Theorem \ref{thm:apriori_estimates}, we establish two auxiliary lemmas. In these lemmas and throughout the proof of the theorem, we drop the subscript of the norm $\Vert \cdot \Vert_{L^2(\mu_1)}$. This improves readability and is justified by Remark \ref{rem:conventions}.
	
	\begin{lemma}
		\label{lem:apriori1}
		Let $f\in\ccinfty(\Omega_1)$ and $\varepsilon>0$, then
		\begin{equation}
			\label{eq:apriori:1}
			\Vert Q^*\nabla^2 f \, Q \Vert^2 \leq \Vert Tf \Vert^2 + \varepsilon \Cq \, \bigl\Vert 	\vert\nabla\phi_1\vert \, \vert Q^*\nabla f\vert \bigr\Vert^2 + C_\varepsilon \Cq \, \Vert Q^*\nabla f \Vert^2,
		\end{equation}
		where $T= \sum_{i,j=1}^{d_1}(QQ^*)_{ij} (\partial_i\partial_j-\partial_i\phi_1\partial_j)$ is the operator from Section \ref{subsubsec:G}.
	\end{lemma}
	\begin{proof}
		First, a simple computation shows that
		\begin{multline*}
			\bigl\Vert Q^*\nabla^2 f \, Q \bigr\Vert_{L^2(\mu_1)}^2 
			= \int_{\Omega_1} \sum_{i,j=1}^{d_2} \Bigl(\sum_{k,l=1}^{d_1} Q_{il}^* \partial_l\partial_k f Q_{kj}\Bigr) \Bigl(\sum_{m,n=1}^{d_1} Q_{in}^* \partial_n\partial_m f Q_{mj}\Bigr) \,\dmui{1} \\
			= \sum_{l,n} \sum_{k,m} (QQ^*)_{ln} \,(QQ^*)_{km} \int_{\Omega_1} \partial_l\partial_k f \partial_n\partial_m f \,\dmui{1}.
		\end{multline*}
		Using integration by parts repeatedly, we obtain that
		\begin{multline*}
			\int_{\Omega_1} \partial_l\partial_k f \partial_n\partial_m f \,\dmui{1} \\
			= \SP{\partial_l\partial_n f-\partial_l f \partial_n\phi_1}{\partial_k\partial_m f-\partial_m f \partial_k \phi_1} {L^2(\mu_1)} 
			- \int_{\Omega_1} \partial_m f\partial_l f \, \partial_k\partial_n \phi_1 \,\dmui{1}.
		\end{multline*}
		Inserting this expression above, we conclude that
		\begin{equation*}
			\bigl\Vert Q^*\nabla^2 f \, Q \bigr\Vert_{L^2(\mu_1)}^2
			= \Vert Tf \Vert_{L^2(\mu_1)}^2 - \int_{\Omega_1} \SP{QQ^* \nabla f}{\nabla^2\phi_1 \,QQ^*\nabla f}{\mathbb{R}^{d_1}} \,\dmui{1}. 
		\end{equation*}
		Finally, assumption ($\Phi_1$4) allows to estimate that
		\begin{multline*}
			\Bigl\lvert \int_{\Omega_1} \SP{QQ^* \nabla f}{\nabla^2\phi_1 \,QQ^*\nabla f}{\mathbb{R}^{d_1}} \,\dmui{1} \Bigr\rvert
			\leq \int_{\Omega_1} \vert QQ^* \nabla f \vert\, \vert \nabla^2\phi_1 \vert \, \vert QQ^*\nabla f \vert \,\dmui{1} \\
			\leq \int_{\Omega_1} (\varepsilon \vert\nabla\phi_1\vert^2+C_\varepsilon) \vert QQ^*\nabla f\vert^2 \,\dmui{1} 
			\leq \Cq  \,\bigl( \varepsilon \,\bigl\Vert \vert\nabla\phi_1\vert \, \vert Q^*\nabla f\vert \bigl \Vert^2 +\, C_\varepsilon \,\Vert Q^*\nabla f \Vert^2\bigr),
		\end{multline*}
		which yields the claim.
	\end{proof}
	
	\begin{lemma}
		\label{lem:apriori2}
		Let $f\in \ccinfty(\Omega_1)$ and $\theta\in (0,1)$, then
		\begin{equation}
			\label{eq:apriori:4}
			\bigl\Vert \vert Q^*\nabla f\vert \,\vert\nabla\phi_1\vert \bigl\Vert^2
			\leq \frac{4(1-\theta)^{-2}}{\lmin}\Vert Q^* \nabla^2 f Q \Vert^2 + \frac{2\Ctheta}{1-\theta} \Vert Q^*\nabla f \Vert^2.
		\end{equation}
	\end{lemma}
	\begin{proof}
		Let $g\in C_c^2(\Omega_1)$. Integration by parts yields
		\begin{equation*}
			\Vert g\nabla\phi_1 \Vert^2
			= \int_{\Omega_1} 2g \nabla g \cdot \nabla\phi_1 \,\dmui{1} + \int_{\Omega_1} g^2 \Delta\phi_1 \,\dmui{1}.
		\end{equation*}
		Here, by Cauchy's inequality
		\begin{equation*}
			\Bigl\vert \int_{\Omega_1} 2g \nabla g \cdot \nabla\phi_1 \,\dmui{1} \Bigr\vert  
			\leq 2 \lVert \eta^{-1/2} \nabla g \Vert \, \Vert \eta^{1/2} g \nabla\phi_1 \Vert
			\leq \eta^{-1} \Vert \nabla g \Vert^2 + \eta \Vert g\nabla\phi_1 \Vert^2
		\end{equation*}
		holds for any $\eta\in (0,1-\theta)$, and 
		\begin{equation*}
			\Bigl\vert \int_{\Omega_1} g^2 \Delta\phi_1 \,\dmui{1} \Bigr\vert \leq \int_{\Omega_1} g^2 (\theta \vert \nabla\phi_1 \vert^2 + \Ctheta ) \,\dmui{1} = \theta \Vert g\nabla\phi_1 \Vert^2 + \Ctheta \Vert g \Vert^2.
		\end{equation*}
		Choosing $\eta = \frac{1-\theta}{2}$ and rearranging the terms, we conclude that
		\begin{equation}
			\label{eq:apriori:2}
			\Vert g\nabla\phi_1 \Vert^2 \leq \frac{4}{(1-\theta)^2} \Vert \nabla g\Vert^2 + \frac{2\Ctheta}{1-\theta} \Vert g\Vert^2.
		\end{equation}
		
		Next, let $g\in C^2(\Omega_1)$ be bounded and satisfy $\vert \nabla g\vert \in L^2(\mu_1)$. Then \eqref{eq:apriori:2} still holds, since by dominated convergence the approximating sequence $(\chi_n(\phi_1) g)_{n\in\mathbb{N}}$ from $C_c^2(\Omega_1)$ satisfies 
		\begin{equation*}
			\Vert \chi_n(\phi_1)g\Vert \to \Vert g\Vert, \quad \Vert \chi_{n}(\phi_1)g\nabla\phi_1 \Vert \to \Vert g\nabla\phi_1 \Vert \quad\text{and}\quad \Vert \nabla (\chi_n(\phi_1)g)\Vert \to \Vert \nabla g\Vert
		\end{equation*}
		as $n\to\infty$. Recall Lemma \ref{lem:cutoffs} for the definition of the cutoff functions $(\chi_n(\phi_1))_{n\in\mathbb{N}}$.
		
		Now, let $f\in\ccinfty(\Omega_1)$ and set $g_\delta \defeq (\vert Q^*\nabla f\vert^2 + \delta)^{1/2}$ for $\delta>0$. Then $g_\delta$ lies in $C^2(\Omega_1)$, is bounded, and satisfies $\vert \nabla g_\delta \vert \in L^2(\mu_1)$.
		Consequently,
		\begin{equation}
			\label{eq:apriori:3.5}
			\Vert g_\delta \nabla\phi_1 \Vert^2 \leq \frac{4}{(1-\theta)^2} \Vert \nabla g_\delta \Vert^2 + \frac{2\Ctheta}{1-\theta} \Vert g_\delta\Vert^2.
		\end{equation}
		Since $\vert Q^*\nabla f\vert < g_\delta$, we calculate that
		\begin{equation*}
			\lvert \partial_i g_\delta \rvert  = \frac{1}{g_\delta} \lvert (Q^*\nabla f)\cdot (Q^*\nabla\partial_i f)\rvert \leq \lvert Q^*\nabla\partial_i f \rvert,
		\end{equation*}
		which implies that $\vert \nabla g_\delta \vert^2 \leq \vert Q^*\nabla^2 f\vert^2 \leq \vert Q^* \nabla^2 f Q \vert^2 \, \lmin^{-1}$. 
		With \eqref{eq:apriori:3.5} we conclude that
		\begin{equation*}
			\bigl\Vert \vert Q^*\nabla f\vert \, \vert\nabla\phi_1\vert \bigr\Vert^2 \leq \Vert g_\delta \nabla\phi_1 \Vert^2 \leq \frac{4 (1-\theta)^{-2}}{\lmin} \Vert Q^* \nabla^2 f Q \Vert^2 + \frac{2\Ctheta}{1-\theta} \Vert g_\delta \Vert^2.
		\end{equation*}
		The claim follows since $g_\delta$ converges to $\vert Q^*\nabla f\vert $ in $L^2(\mu_1)$ as $\delta \to 0$.
	\end{proof}
	
	\begin{proof}[Proof of Theorem \ref{thm:apriori_estimates}]
		First, let $f\in \ccinfty(\Omega_1)$. Using equation \eqref{eq:T}, we observe that 
        \begin{align}
			\Vert (I-\konst T) f\Vert^2 &= \Vert f\Vert^2 - 2\konst\SP{Tf}{f}{L^2(\mu_1)} + \konst^2 \Vert Tf \Vert^2 \notag \\
			&= \Vert f \Vert^2 + 2\konst\Vert  Q^*\nabla f\Vert^2 + \konst^2 \Vert Tf \Vert^2. \label{eq:apriori:H56}
        \end{align}
		To obtain the first estimate, we let $\theta\in (0,1)$, $\varepsilon\in (0,\emax)$ and apply \eqref{eq:apriori:4} to the right-hand side of \eqref{eq:apriori:1}. Rearranging terms, this yields
		\begin{equation*}
			\label{eq:apriori:n1}
			a_1 \Vert  Q^* \nabla^2 f  Q \Vert^2
			\leq \Vert Tf \Vert^2  + a_2 \Vert  Q^*\nabla f \Vert^2 \\
			\leq \max\Bigl(\frac{1}{\konst^{2}},\frac{a_2}{2\konst}\Bigr) \bigl(\konst^2 \Vert Tf\Vert^2 + 2\konst \Vert  Q^*\nabla f\Vert^2\bigr),
		\end{equation*}
		where $a_1 = 1-\varepsilon/\emax>0$ and $a_2 = (2 \varepsilon \Ctheta (1-\theta)^{-1} + C_\varepsilon) \Cq$.
		With \eqref{eq:apriori:H56} we conclude that
		\begin{equation}
            \label{eq:apriori:n1+}
			\Vert  Q^* \nabla^2 f  Q \Vert^2 
			\leq a_1^{-1} \max\Bigl(\frac{1}{\konst^{2}},\frac{a_2}{2\konst}\Bigr) \Vert (I-\konst T) f\Vert^2 = \xi_{\varepsilon,\theta}^{(1)} \Vert (I-\konst T) f\Vert^2.
		\end{equation}
        If instead we apply \eqref{eq:apriori:1} to the right-hand side of \eqref{eq:apriori:4}, after rearranging we obtain
        \begin{align*}
			\label{eq:apriori:n2}
			a_1 \bigl\Vert \vert Q^*\nabla f\vert \vert\nabla\phi_1\vert \bigr\Vert^2 
			&\leq a_3 \Vert Tf \Vert^2 + a_4 \Vert  Q^*\nabla f \Vert^2 \\
			&\leq \max\Bigl(\frac{a_3}{\konst^{2}},\frac{a_4}{2\konst}\Bigr) \bigl(\konst^2 \Vert Tf\Vert^2 + 2\konst \Vert  Q^*\nabla f\Vert^2\bigr),
		\end{align*}
		where $a_3 = 4 \lmin^{-1} (1-\theta)^{-2} = (\emax \Cq)^{-1}$ and $a_4= 2\Ctheta (1-\theta)^{-1} + C_\varepsilon / \emax $.
 		Again using \eqref{eq:apriori:H56}, we get
		\begin{equation}
            \label{eq:apriori:n2+}
			\bigl\Vert \vert \nabla\phi_1\vert\, \vert Q^* \nabla f\vert \bigr\Vert^2
			\leq a_1^{-1} \max\Bigl(\frac{a_3}{\konst^{2}},\frac{a_4}{2\konst}\Bigr) \Vert (I-\konst T) f\Vert^2 = \xi_{\varepsilon,\theta}^{(2)} \Vert (I-\konst T) f\Vert^2.
		\end{equation}
		Taking the infimum in \eqref{eq:apriori:n1+} and \eqref{eq:apriori:n2+} proves the  estimates \eqref{eq:aprioriestimate:1} and \eqref{eq:aprioriestimate:2} for $f\in \ccinfty(\Omega_1).$
        
		Now, let $f\in C_c^2(\Omega_1)$. By convolution with an approximate identity, there is a sequence $(f_n)_{n\in\mathbb{N}}$ in $\ccinfty(\Omega_1)$ such that $\partial^\alpha f_n\to \partial^\alpha f$ uniformly as $n\to\infty$ for all multi-indices $\alpha$ of order $\vert\alpha\vert\leq 2$ and such that the support of $f$ and every $f_n$ is contained in a fixed compact set $K\subset \Omega_1$. On $K$ the continuous functions $\partial_i\phi_1$ are bounded, so
		\begin{equation}
			\label{eq:T:altcore}
			Tf_n = \sum_{i,j=1}^{d_1}( Q Q^*)_{ij} (\partial_i\partial_jf_n-\partial_i\phi_1\partial_jf_n) \to 
			\sum_{i,j=1}^{d_1}( Q Q^*)_{ij} (\partial_i\partial_jf-\partial_i\phi_1\partial_jf) =
			Tf
		\end{equation}
		uniformly as $n\to\infty$. Similarly, $(I-\konst T)f_n \to (I-\konst T)f$, 
		\begin{equation*}
			\lvert\nabla\phi_1\rvert^2 \lvert  Q^* \nabla f_n \rvert^2 \to \lvert\nabla\phi_1\rvert^2 \lvert  Q^* \nabla f \rvert^2 \qquad\text{and}\qquad
			\lvert  Q^* \nabla^2 f_n \, Q \rvert^2 \to 	\lvert  Q^* \nabla^2 f \, Q \rvert^2
		\end{equation*}
		uniformly as $n\to\infty$.
		Since $\mu_1$ is finite, uniform convergence implies convergence in $L^2(\mu_1)$. Hence the estimates \eqref{eq:aprioriestimate:1} and \eqref{eq:aprioriestimate:2} carry over to $f\in C_c^2(\Omega_1)$.
	\end{proof}

	\begin{theorem}[Verification of (H4), 2nd inequality]
		\label{thm:H4:IE2}
		The estimate $\Vert BA(I-P)f \Vert_H \leq N_2 \Vert (I-P)f \Vert_H$ holds for all $f\in \core$ and
		\begin{equation}
			\label{eq:N2}
			N_2 = \xi_1^{1/2} \bigl\Vert \vert\nabla\phi_2\vert^2\bigr\Vert_{L^2(\mu_2)} + \xi_2^{1/2} \Vert  Q\nabla^2\phi_2\Vert_{L^2(\mu_2)}.
		\end{equation}
	\end{theorem}
	
	\begin{proof}
		Since the hypocoercivity condition (G) has been verified in Theorem \ref{thm:G_selfadjoint}, we can use Lemma \ref{lem:H4:suffcrit}. Thus it suffices to show that $AP(\core)\subset D(A^*)$ and 
		\begin{equation}
        \label{eq:H4IE2 goal}
			\Vert A^*APf\Vert \leq N_2 \Vert (I-G)f\Vert \qquad\text{for all }f\in \core.
		\end{equation}
		
		Let $f\in \core$. As $(A,D(A))$ is antisymmetric, $(A^*,D(A^*))$ extends $(-A,D(A))$. Hence $APf\in D(A)\subset D(A^*)$ by Lemma \ref{lem:D5D7}, and $A^*APf = -A^2Pf$ is given by \eqref{eq:AAP}. Thus we obtain that
		\begin{equation*}
			\lVert A^* A P f\rVert 
			\leq \bigl\lVert \lvert \nabla\phi_2 \rvert^2 \cdot \lvert Q^* \gradx^2( P_S f) \, Q\rvert \bigr\rVert
			+ \bigl\lVert \lvert Q\nabla^2\phi_2 \rvert \cdot \lvert \nabla\phi_1\rvert \lvert Q^* \gradx P_S f\rvert \bigr\rVert.
		\end{equation*}
		Using Fubini's theorem, the right-hand side can be identified as
		\begin{equation*}
			\bigl\Vert  Q^*\gradx^2( P_S f) \, Q \bigr\Vert_{L^2(\mu_1)}
			\bigl\Vert \vert\nabla\phi_2\vert^2\bigr\Vert_{L^2(\mu_2)} 
			+ \bigl\Vert \vert \nabla\phi_1\vert \vert Q^*\gradx P_S f\vert \bigr\Vert_{L^2(\mu_1)} 
			\bigl\Vert  Q\nabla^2\phi_2\bigr\Vert_{L^2(\mu_2)},
		\end{equation*}
		so the a priori estimates from Theorem \ref{thm:apriori_estimates} yield
		\begin{equation*}
			\label{eq:H4:2:int}
				\Vert A^* APf \Vert
				\leq \bigl( \xi_1^{1/2} \bigl\Vert \vert\nabla\phi_2\vert^2\bigr\Vert 
				+  \xi_2^{1/2} \Vert  Q\nabla^2\phi_2\Vert
                \bigr)
				 \;\Vert (I-\konst T) P_S f \Vert.
		\end{equation*}
        Therefore, \eqref{eq:H4IE2 goal} follows by applying the estimate \eqref{eq:I-kTvsI-G}.
	\end{proof}
	
	\subsection{Main Result and Convergence Rate}
	\label{subsec:main_result}
	
	We are now able to apply Theorem \ref{thm:hypocoercivity:abstract}, the abstract hypocoercivity theorem, to conclude our main result. The computation of the convergence rate is analogous to \cite{Bertram_Grothaus_HypocoercivityLangevinMultipicativeNoise}.
	\begin{proof}[Proof of Theorem \ref{thm:main:abstract}]
		Below, as in Section \ref{subsec:hypocoercivity_conditions}, we assume that $\phi_2$ is radially symmetric, i.e., $\tau=I$ and $b=0$. The general result is then obtained via the coordinate transformation $y\mapsto \tau^{-1}(y+b)$ as detailed in Appendix~B.
        
		If $\tau=I$ and $b=0$, it follows from Section \ref{subsec:hypocoercivity_conditions} that the hypocoercivity conditions are satisfied for constants $\Lambda_m$, $\Lambda_M$, $N_1$ and $N_2$ as in \eqref{eq:Lambda_m}, \eqref{eq:Lambda_M}, \eqref{eq:N1} and \eqref{eq:N2}. Hence Theorem \ref{thm:hypocoercivity:abstract} applies. To specify the convergence rate, we choose explicit values for $\varepsilon, \delta \in (0,1)$ and $\kappa \in (0,\infty)$ satisfying \eqref{eq:convrate:con1and2}.
		
		Note that
		\begin{equation}
			\label{eq:sigma_dependence}
			\Lambda_m = \Cell \Lambda_m^\prime \qquad\text{and}\qquad N_1 = \Csigma N_1^\prime
		\end{equation}
		for positive constants $\Lambda_m^\prime$ and $N_1^\prime$ independent of $\Sigma$, while $\Lambda_M$ and $N_2$ do not depend on $\Sigma$ at all.
		Setting 
		\begin{equation*}
			\delta \defeq \frac{\Lambda_M}{1+\Lambda_M} \frac{1}{1+N_1+N_2} \in (0,1),
		\end{equation*} 
		the inequalities in \eqref{eq:convrate:con1and2} take the form
		\begin{equation}
			\label{eq:conv_crit_simple}
			\Cell \, \Lambda_m^\prime - \varepsilon \, r(\Csigma) \geq \kappa \qquad\text{and} \qquad\varepsilon \, s \geq \kappa
		\end{equation}
		with strictly positive constants
		\begin{equation*}
			r(\Csigma) \defeq (1+\Csigma N_1^\prime + N_2) \big(1 + \tfrac{1+\Lambda_M}{2\Lambda_M}(1+\Csigma N_1^\prime + N_2)\bigr)
			\quad\text{and}\quad 
			s\defeq\tfrac{1}{2} \tfrac{\Lambda_M}{1+\Lambda_M}.
		\end{equation*}
        
		Define $a_1, a_2, a_3 \in (0,\infty)$ via
		$r(\Csigma)+s \eqdef a_1 + a_2\Csigma + a_3 \Csigma^2.$ Then
		\begin{equation*}
			\bar{\varepsilon}(\Csigma) \defeq \frac{\Lambda_m^\prime \Csigma}{r(\Csigma)+s} = \frac{\Lambda_m^\prime \Csigma}{a_1 + a_2 \Csigma + a_3 \Csigma^2}
		\end{equation*}
		is bounded as a function of $\Csigma \in (0,\infty)$, so $\bar{\varepsilon}_\text{max} \defeq \max(1,\sup_{\alpha>0} \bar{\varepsilon}(\alpha))$ is finite.
        The ellipticity constant $\Cell$ is bounded by the diagonal entries of $\Sigma$. In particular, $\Cell \leq \Csigma$.
		  Hence for $C\in(1,\infty)$, we can set
		\begin{equation*}
			\varepsilon \defeq \frac{C-1}{C} \frac{\bar{\varepsilon}(\Csigma)}{\bar{\varepsilon}_\text{max}} \frac{\Cell}{\Csigma} \in (0,1), 
            \quad\text{then}\quad 
            \varepsilon(r(\Csigma)+s) = \frac{C-1}{C} \frac{1}{\bar{\varepsilon}_\text{max}} \Lambda_m^\prime \Cell \leq \Lambda_m^\prime \Cell.
		\end{equation*}
        Therefore $\kappa \defeq \varepsilon s \leq \Lambda_m^\prime \Cell - \varepsilon r(\Csigma)$ satisfies \eqref{eq:conv_crit_simple}, and Theorem \ref{thm:hypocoercivity:abstract} yields for all $t\geq 0$ and $g\in H$ that
		\begin{equation}
            \label{eq:exponential_conv_kappa}
			\Vert T_t g - \SPH{g}{1} \Vert \leq \kappa_1 e^{-\kappa_2t} \Vert g - \SPH{g}{1} \Vert,
		\end{equation} where $\kappa_1 = (\frac{1+\varepsilon}{1-\varepsilon})^{1/2}$ and $\kappa_2 = \frac{\kappa}{1+\varepsilon}$.
		Since $\varepsilon \leq \frac{C -1}{C}$ and $(C - 1)^2 \geq 0$, we observe that
		\begin{equation*}
			\kappa_1^2 = \frac{1+\varepsilon}{1-\varepsilon} \leq\frac{{(2C - 1)}/{C}}{{1}/{C}} = 2C - 1 \leq C^2\quad \text{and}\quad \kappa_2 = \frac{\kappa}{1+\varepsilon} > \frac{\kappa}{2}\eqdef \lambda.
		\end{equation*}
		Consequently, \eqref{eq:exponential_conv_kappa} remains valid if $\kappa_1$ and $\kappa_2$ are replaced by $C$ and $\lambda$ respectively.
		This concludes the proof as for $n_i \defeq {2 \bar{\varepsilon}_\text{max}}({s \Lambda_m^\prime)}^{-1} a_i$ the rate $\lambda$ assumes the form
		\begin{equation*}
			\lambda 
			= \frac{s\varepsilon}{2} = \frac{s}{2} \frac{C-1}{C \bar{\varepsilon}_\text{max}} \frac{\Lambda_m^\prime \Cell}{a_1 + a_2 \Csigma + a_3 \Csigma^2} = \frac{C-1}{C} \frac{\Cell}{n_1 + n_2 \Csigma + n_3 \Csigma^2}. \qedhere
		\end{equation*}
	\end{proof}
	
	\section{Sufficient Criteria for Essential \texorpdfstring{$m$}{m}-Dissipativity of \texorpdfstring{$\sL$}{L} and \texorpdfstring{$T$}{T}}
	\label{sec:crit_ess-m-diss}
	Next, we present some explicit sufficient criteria for the abstract conditions (eT) and (eL).
	\begin{lemma}
	    \label{lem:essT:suff}
	    Let $d_1 \leq d_2$ and $Q\in \mathbb{R}^{d_1\times d_2}$ be of full rank. Assume \textup{($\Phi_1$1)}. Then 
        $\textup{(eT)}$ holds
        if \textup{($\Phi_1$5)} or the following condition is satisfied:
    	\begin{itemize}[leftmargin = 1.1cm]
    		\item[\textup{(eT)$'$}] It holds $\vert\nabla\phi_1\vert \in L^2(\mu_1)$, $\vert\nabla\phi_1\vert^2 +\vert\nabla^2\phi_1\vert \in L_\textup{loc}^2(\mathbb{R}^{d_1},\mu_1)$, $\vert \nabla\phi_1\vert^2 e^{-\phi_1}\in L_\textup{loc}^\infty(\mathbb{R}^{d_1},\mu_1)$ and $d_1\geq 2$.
    	\end{itemize}
	\end{lemma}
	
	\begin{proof}
		As detailed in Appendix \ref{app:details_coordinate_trafos}, after a suitable coordinate transformation we may assume that $(T,C_c^2(\Omega))$ has the simple form
		\begin{equation*}
			\label{eq:T:simple}
			T = \Delta - \nabla \phi_1 \cdot \nabla.
		\end{equation*}
		
		As a densely defined, dissipative operator $(T,C_c^2(\Omega_1))$ has a dissipative closure $(T,D(T))$. We next show that 
		\begin{equation}
			\label{eq:T:ext}
			(\Delta - \nabla\phi_1 \cdot \nabla,\ccinfty(\mathbb{R}^{d_1}))\subset (T,D(T)).
		\end{equation} 
		If \textup{($\Phi_1$5)} holds, this is trivial since $\mathbb{R}^{d_1}=\Omega_1$. If \textup{(eT)$'$} holds instead, define $f_n\defeq f \chi_n(\phi_1)\in C_c^2(\Omega_1)$ for $f\in \ccinfty(\mathbb{R}^{d_1})$, and observe that
		\begin{multline*}
			T f_n = (\Delta f - \nabla \phi_1 \cdot \nabla f) \chi_n(\phi_1) + (\nabla f \cdot \nabla\phi_1) \chi_n'(\phi_1) \\
			+ f( \vert\nabla\phi_1\vert^2 \chi_n''(\phi_1) - \vert\nabla\phi_1\vert^2 \chi_n'(\phi_1) + \Delta \phi_1 \chi_n'(\phi_1))
			\to \Delta f - \nabla \phi_1 \cdot \nabla f 
		\end{multline*}
		in $L^2(\mu_1)$ by dominated convergence using that $\vert\nabla\phi_1\vert^2, \vert\nabla^2\phi_1\vert \in L_\textup{loc}^2(\mathbb{R}^{d_1},\mu_1)$.
		Since $(T,D(T))$ is closed, \eqref{eq:T:ext} follows.
		
		Now, the proof of Theorem \ref{thm:apriori_estimates}, primarily equation \eqref{eq:T:altcore}, shows that the subset $\ccinfty(\Omega_1)$ of $\ccinfty(\mathbb{R}^{d_1})$ is a core of $(T,C_c^2(\Omega_1))$. Together with \eqref{eq:T:ext}, this implies that $(T,D(T))$ is the closure of 
		\begin{equation*}
			\label{eq:T:eberle}
			(\Delta - \nabla\phi_1 \cdot \nabla,\ccinfty(\mathbb{R}^{d_1})).
		\end{equation*}
        
        We may therefore apply standard criteria for this operator: \cite[Theorem 3.1]{Wielens_EssSelfAdjointSchroedingerOperators} under assumption \textup{($\Phi_1$5)}, and \cite[Theorem 2.5]{Eberle_EssSelfAdjointSchroedingerOperators} under assumption \textup{(eT)$'$}.
		By the equivalences recalled in Remark~\ref{rem:ess-m-dissipativity}, both statements yield that $(T,C_c^2(\Omega_1))$ is essentially $m$-dissipative.
	\end{proof}
	
	\begin{lemma}
		\label{lem:essL:suff}
		Let $d\defeq d_1 = d_2 \in \mathbb{N}$ with $d\geq 2$, let $Q \in \mathbb{R}^{d\times d}$ be invertible, and let $\Sigma \in \mathbb{R}^{d\times d}$ be symmetric and positive definite. Assume \textup{($\Phi_1$1)} and \textup{($\Phi_1$5)}, as well as \textup{($\Phi_2$1)}, \textup{($\Phi_2$3)}, \textup{($\Phi_2$6)}, and \textup{($\Phi_2$7)}. Then \textup{(eL)} is satisfied.
	\end{lemma}
	
	\begin{proof}
        As detailed in Appendix \ref{app:details_coordinate_trafos}, after a suitable coordinate transformation we may assume that $(\sL, \core)$ has the simple form
		\begin{equation*}
			\label{eq:L:simple}
			\sL = \Delta_y - \nabla \phi_2 \cdot \grady + \nabla \phi_2 \cdot \gradx - \nabla \phi_1 \cdot \grady.
		\end{equation*}
		Due to \textup{($\Phi_2$1)}, \textup{($\Phi_2$3)} and \textup{($\Phi_2$7)}, \cite[Theorem 2.5]{Eberle_EssSelfAdjointSchroedingerOperators} yields that $(\Delta - \nabla\phi_2 \cdot \nabla,\ccinfty(\Omega_2))$ is essentially self-adjoint on $L^2(\mu_2)$. By \textup{($\Phi_1$1)} and \textup{($\Phi_1$5)}, $\phi_1$ is continuously differentiable and thus locally Lipschitz on $\Omega_1=\mathbb{R}^{d_1}$. With these two observations, \textup{($\Phi_2$1)}, and \textup{($\Phi_2$6)}, we can apply \cite[Theorem 2.9]{Nonnenmacher_Grothaus_EssSelfAdjointGenerator} to obtain that $(\sL, \core)$ is essentially $m$-dissipative.
	\end{proof}

    \begin{remark}[The one-dimensional case]
        \label{rem:one dimensional case}
        In dimension $d_i=1$ explicit criteria for essential self-adjointness of $(\Delta - \nabla\phi_i \cdot \nabla,\ccinfty(\mathbb{R}^{d_i}))$ can be found in \cite[Section~2]{Wielens_EssSelfAdjointSchroedingerOperators}. However, the resulting conditions depend on the boundary behavior of $\phi_i$ and are thus less transparent than the integrability assumptions required by \cite[Theorem 2.5]{Eberle_EssSelfAdjointSchroedingerOperators}. Since our list of assumptions in Corollary \ref{thm:main:explicit} is already rather involved, we do not pursue the one-dimensional case here.
        Nevertheless, we emphasize that a sufficient condition for (eT) that is explicit in $\phi_1$ and allows singularities can be obtained for $d_1=1$ in this way. Similarly, under such explicit assumptions on $\phi_2$ ensuring self-adjointness of $(\Delta - \nabla\phi_2 \cdot \nabla,\ccinfty(\mathbb{R}^{d_2}))$ for $d_2=1$, the proof of Lemma \ref{lem:essL:suff} yields an explicit criterion for (eL) for $d=1$.
    \end{remark}
    Using the above criteria, we obtain the explicit version of our main theorem.
    \begin{proof}[Proof of Corollary \ref{thm:main:explicit}]
		Since $\Sigma$ is a constant, symmetric, and positive definite matrix, ($\Sigma$1) and ($\Sigma$2) are clearly satisfied for   $\Cell = \lambda_\textup{min}(\Sigma)$ and $\Csigma = \vert \Sigma \vert$.
		Due to Lemma \ref{lem:essT:suff} and \ref{lem:essL:suff}, assumptions (eT) and (eL) hold. Thus \textup{($\Phi_1$1)} - \textup{($\Phi_1$4)}, \textup{($\Phi_2$1)} - \textup{($\Phi_2$5)}, \textup{($\Sigma$1)}, \textup{($\Sigma$2)}, \textup{(eL)} and \textup{(eT)} are all satisfied, and the claim follows from Theorem \ref{thm:main:abstract}.
	\end{proof}

    \section{Stochastic Interpretation}
	\label{sec:stochastic_interpretation}
 
    In this section, we identify the semigroup $(T_t)_{t\geq 0}$ in  Theorem \ref{thm:main:abstract} with the transition semigroup of the diffusion process determined by the unique weak solutions of SDE  \eqref{eq:sde:general} from Theorem 
    \ref{thm:existence_uniqueness_weak_sol}, as stated in Theorem \ref{thm:identification}.
    To this end, we establish the existence of a conservative diffusion process associated with $(\cL ,D(\cL))$ and show that this process yields martingale and weak solutions to the SDE \eqref{eq:sde:general}. 
    To keep the presentation focused, some details are deferred to Appendix \ref{app:details_stochastic_representation}.
   
	\begin{standingassumptions}
	    Throughout this section, let $\Sigma\colon \mathbb{R}^{d_2} \to \mathbb{R}^{d_2\times d_2}$ be pointwise symmetric with ($\Sigma1$) and ($\Sigma2$). Let $\phi_1$ and $\phi_2$ be $\R\cup \{+\infty\}$-valued potentials on $\mathbb{R}^{d_i}$ satisfying ($\Phi_i$1) and ($\Phi_i$3), respectively. Additionally, assume (eL).
	\end{standingassumptions}
    
	Under these assumptions, which are not intended to be minimal, the definitions and results from Section \ref{subsec:data_conditions} apply. In particular, all data conditions are satisfied.
	
	\subsection{Construction of the Associated Process}

   An associated process is obtained from the theory of generalized Dirichlet forms \cite{Stannat_GeneralizedDirichletForms}. Due to (D7), it is conservative, so we can view it as a process on the subspace $C([0,\infty); \Omega)$ of $\sZ$. We denote the induced $\sigma$-field on $C([0,\infty);\Omega)$ by $\mathcal{B}$ 
    and the raw natural filtration by $(\mathcal{G}_t)_{t\geq 0}$.
    
    The notions of nests, exceptional sets, and properties holding quasi-everywhere are in the following understood with respect to the sub-Markovian strongly continuous contraction resolvent $(G_\alpha)_{\alpha > 0}$ defined by $(\cL,D(\cL))$, see Appendix \ref{app:details_stochastic_representation} for details.
	
	\begin{theorem}
		\label{thm:associated_process}
		There exists a conservative diffusion process $\mathbf{M}=(C([0,\infty); \Omega), \mathcal{B}, (\mathcal{G}_t)_{t\geq 0},$
        $(X_t, Y_t)_{t\geq 0}, (\mathbb{P}_{(x,y)})_{(x,y)\in \Omega})$ with state space $\Omega$ that is properly associated in the resolvent sense with $(\cL,D(\cL))$,  i.e., $R_\alpha f$ is a quasi-continuous $\mu$-version of $G_\alpha f$ for all $\alpha>0$ and all bounded $f \in L^2(\Omega,\mu)$. Here, $(R_\alpha)_{\alpha >0}$ denotes the resolvent of $\mathbf{M}$.
    \end{theorem}
    
	\begin{proof}
		In Appendix \ref{app:details_stochastic_representation} it is verified that $(T_t)_{t\geq 0}$ is a sub-Markovian strongly continuous contraction semigroup on $L^2(\Omega, \mu)$, that its generator $(\cL,D(\cL))$ has a core of continuous bounded functions forming an algebra, that there is a sequence of continuous functions in $D(L)$ separating the points of $\Omega$ and that a nest of compact sets exists. This yields the existence of a special standard process $\mathbf{M}$ properly associated in the resolvent sense with $(T_t)_{t\geq 0}$, see \cite[Theorem IV.2.2]{Stannat_GeneralizedDirichletForms}.
		
		The locality of the generator and the conservativity of $(T_t)_{t\geq 0}$ imply $\mathbb{P}_{(x,y)}$-a.s.{} continuous paths and $\mathbb{P}_{(x,y)}$-a.s.{} infinite lifetime for quasi-all $(x,y)\in \Omega$. This is obtained by \cite[Theorem 3.3]{Trutnau} in combination with \cite[Theorem IV.3.8(ii)]{Stannat_GeneralizedDirichletForms} for quasi-all $(x,y)\in \Omega$. By the procedure of restriction and trivial extension, see  \cite[Remark IV.6.2(i), Corollary IV.6.5 and Remark IV.3.23(i)]{MR92},
		 one can modify the process for an exceptional Borel set of initial points to obtain the statement for all $(x,y)\in \Omega$. It follows that $\mathbf{M}$ is even a conservative diffusion process.
	\end{proof}
	
	\subsection{Solution of the Martingale Problem}
    
    Next, we observe that $\mathbf{M}$ yields solutions to the martingale problem for $(\sL,\core)$ with certain absolutely continuous initial distributions. 
    This follows from \cite[Lemma 5.1]{Conrad_2010}.
    
    \begin{theorem}
		\label{thm:martingale_problem_path_space}
        Let $h \in L^2(\mu)$ be a probability density with respect to $\mu$ and set $\mathbb{P}_{h\mu}\defeq \int_{\Omega} \mathbb{P}_{(x,y)}(\cdot) \, h(x,y) \,\textup{d}\mu(x,y)$.
        Then $\mathbb{P}_{h\mu}$ solves the martingale problem for $(\cL,D(\cL))$, i.e., for all $f\in D(\cL)$, the process $(M_t^{[f]})_{t\geq 0}$ on $(C([0,\infty);\Omega), \mathcal{B}, \mathbb{P}_{h\mu})$ defined by
		\begin{equation}
			\label{eq:martingale:m}
			M_t^{[f]} \defeq f(X_t, Y_t) - f(X_0, Y_0) - \int_{0}^{t} \cL f(X_s, Y_s) \,\textup{d}s, \quad t \ge 0,
		\end{equation}
		is an $(\mathcal{G}_t)_{t\geq 0}$-martingale. 
        Furthermore, if $f\in D(\cL)$ with $f^2\in D(\cL)$ and $\cL f \in L^4(\mu)$, then also the process $(N_t^{[f]})_{t\geq 0}$ given by
		\begin{equation}
			\label{eq:martingale:n}
			N_t^{[f]} \defeq (M_t^{[f]})^2 - \int_{0}^{t} \bigl(\cL(f^2)-2f \cL f\bigr)(X_s, Y_s) \,\textup{d}s, \quad t \ge 0,
		\end{equation}
		is a martingale with respect to $(\mathcal{G}_t)_{t\geq 0}$, identifying the quadratic variation of $(M_t^{[f]})_{t\geq 0}$.
	\end{theorem}
	
    \begin{remark}
        \label{rem:martingale_solutions:regularity}
        
        For $f\in\core=C_c^2(\Omega)$, the functions $\sL f$ and $\sL (f^2)$ are bounded and, in particular, lie in $L^4(\mu)$. We may therefore consider the continuous representative of $f$ and bounded representatives of $\cL f$ and $\cL (f^2)$ in the definition of $M^{[f]}$ and $N^{[f]}$ to obtain continuous versions of these martingales that satisfy \eqref{eq:martingale:m} and \eqref{eq:martingale:n} pointwise. In particular,  for such $f$, $M^{[f]}$ and $N^{[f]}$ are also martingales with respect to the minimal right-continuous filtration $(\mathcal{G}_{t+})_{t\geq 0}$ where $\mathcal{G}_{t+}= \cap_{s > t} \mathcal{G}_s$.
	\end{remark}

    \subsection{Identification as the Weak Solution to the SDE}

    In view of Theorem \ref{thm:martingale_problem_path_space}, we can follow the classical reconstruction based on L\'evy's characterization of Brownian motion to obtain weak solutions to SDE \eqref{eq:sde:general}. The details are deferred to Appendix \ref{app:details_stochastic_representation}.

	\begin{theorem}
		\label{thm:weak_solution_sde}
		Let $h \in L^2(\mu)$ be a probability density with respect to $\mu$ and $\sigma\colon \Omega_2 \to \mathbb{R}^{d_2 \times d_2}$ be continuous with $\Sigma = \sigma \sigma^*$.
        Then there exists a $d_2$-dimensional $(\mathcal{G}_{t+})_{t\geq 0}$-adapted Brownian motion $(B_t)_{t\geq 0}$ on
        $(C([0,\infty);\Omega), \mathcal{B}, \mathbb{P}_{h\mu})$ such that $\mathbb{P}_{h\mu}$-a.s. for all $t\geq 0$
        \begin{align}
            \label{eq:SDE:X}
			X_t - X_0 &= \int_{0}^{t} Q \nabla\phi_2 (Y_s) \,\textup{d}s, \\
            \label{eq:SDE:Y}
			Y_t - Y_0 &= \int_{0}^{t} (\divergence \Sigma - \Sigma \nabla\phi_2)(Y_s) - Q^*\nabla\phi_1 (X_s) \,\textup{d}s + \int_{0}^{t} \sqrt{2} \sigma(Y_s) \,\textup{d}B_s.
        \end{align}
        In particular, the diffusion process $(X, Y)$ under $\mathbb{P}_{h\mu}$ coincides with the unique weak solution to SDE \eqref{eq:sde:general} with initial distribution $h\,\textup{d}\mu$ obtained in Theorem \ref{T:2.4}. 
	\end{theorem}

    As a consequence, we finally obtain the stochastic representation of $(T_t)_{t\geq 0}$ as the transition semigroup of the unique weak solution to SDE \eqref{eq:sde:general}.
    
    \begin{proof}[Proof of Theorem \ref{thm:identification}]
        By a combination of Theorems \ref{thm:associated_process} and \ref{thm:weak_solution_sde}, the unique weak solution to SDE \eqref{eq:sde:general} with initial distribution $h\,\textup{d}\mu$ is a conservative diffusion process properly associated with $(\sL,D(\cL))$ in the resolvent sense. By \cite[Lemma II.4.2]{MR92} it follows that $R_\alpha f$ is a $\mu$-version of $G_\alpha f$  even for all $f\in L^2(\Omega,\mu)$ and $\alpha >0$. This in turn is equivalent to $p_t f$ being a $\mu$-version of $T_tf$, for all $t>0$ and $f\in L^2(\Omega,\mu)$, see \cite[Exercise IV.2.7]{MR92}.
    \end{proof}

    \section*{Acknowledgments}
    Part of this article is based on the third-named author’s Master’s thesis \cite{Pfohl2024}.
    
    The research of Zhen-Qing Chen is partially supported by a Simons Foundation fund. The research of Onno Pfohl is funded
    in part by the Deutsche Forschungsgemeinschaft (DFG, German Research Foundation) under Germany's Excellence Strategy - The Berlin Mathematics Research Center MATH+ (EXC-2046/1, EXC-2046/2, project ID: 390685689).

	\printbibliography[title={References}]

    \appendix

    \section{Deferred Proofs from Sections \texorpdfstring{\ref{sec:intro}}{1}, \texorpdfstring{\ref{sec:hypocoercivity_abstract}}{3} and \texorpdfstring{\ref{sec:hypocoercivity_application}}{4}}
    \label{app:deferred_proofs}
    Before we present the proofs of Proposition \ref{prop:vonNeumann}, Lemma \ref{lem:D5D7}, and Lemma \ref{lem:phi2_integrals}, we restate and prove the incompatibility of the growth condition \eqref{eq:growth:old} with singular potentials, as claimed in the introduction.

    \begin{proposition}
        \label{prop:growth}
        Let $\phi_1\colon \mathbb{R}^{d_1} \to \mathbb{R}\cup\{\infty\}$ be continuous in the extended sense.  Assume that $\phi_1 \in C^2(\Omega_1)$, where $\Omega_1 = \{ \phi_1 < \infty\}$, and that there is $C\in (0,\infty)$ such that 
        \begin{equation}
            \label{eq:growth:old:app}
            \vert \nabla^2 \phi_1(x)\vert \leq C(\vert \nabla \phi_1(x)\vert + 1) \qquad\textup{for all }x\in \Omega_1.
        \end{equation}
        Then, either $\Omega_1=\emptyset$ or $\Omega_1=\R^{d_1}.$
    \end{proposition}

    \begin{proof}
        Assume that $\Omega_1\neq \emptyset$. Fix $x\in \Omega_1$ and $y\in \R^{d_1}$. We show that $y\in \Omega_1$.

        Define the path $\gamma(t)= (1-t)x+ty$, $t\in [0,1]$ and $\mathcal{T} \defeq \{ t \in [0,1]:\gamma(t) \in \Omega_1 \}$. By continuity of $\phi_1$ and $\gamma$, $\Omega_1$ is open in $\R^{d_1}$ and $\mathcal{T}$ is relatively open in $[0,1]$. Since $0\in \mathcal{T}$, it follows that $T \defeq \sup \{t \in [0,1]: [0,t] \subset \mathcal{T} \}>0$. For $t\in [0,T)$, define $g(t)\defeq \nabla \phi_1(\gamma(t))$. Then \eqref{eq:growth:old:app} yields 
        \begin{equation*}
            \vert g'(t)\vert = \vert \nabla^2 \phi_1(\gamma(t)) (y-x)\vert \leq C ( \vert \nabla \phi_1(\gamma(t))\vert + 1) \vert y -x\vert = C\vert y-x\vert (\vert g(t)\vert +1). 
        \end{equation*}
        It follows that
        \begin{equation*}
            \vert g(t)\vert \leq \vert g(0)\vert +  C \vert y-x\vert \int_0^t (\vert g(s)\vert +1)\,\textup{d}s\qquad\textup{for all }t\in [0,T).
        \end{equation*}
        Applying Gronwall's lemma to $\vert g(t)\vert +1$, we obtain that
        \begin{equation*}
            \vert g(t)\vert +1 \leq (\vert g(0)\vert +1) \exp( C \vert y-x\vert \,t ) \qquad\textup{for all }t\in [0,T).
        \end{equation*}
        Since $\frac{d}{dt}\phi_1(\gamma(t)) = g(t) \cdot (y-x)$, it follows that $\phi_1$ is uniformly bounded on $[0,T)$. By continuity $\phi_1(\gamma(T))<\infty$, i.e., $T \in \mathcal{T}$. Since $\mathcal{T}$ is relatively open, it follows that $T=1$. Thus, $1 \in \mathcal{T}$, i.e., $ y \in \Omega_1$. This shows that $\Omega_1=\R^{d_1}$.
    \end{proof}

    \begin{proof}[Proof of Proposition \ref{prop:vonNeumann}]
		Although formulated for complex Hilbert spaces \cite[Proposition X.4.2~a)-c)]{Conway} (and the underlying \cite[Lemma X.1.7]{Conway}) apply to real Hilbert spaces with the same proof. Consequently, $(I+T^*T)\colon D(T^*T) \to H$ is bijective with bounded inverse and $B\defeq T^*(I+TT^*)^{-1}$ is well-defined on $H$ and bounded by 1. The claim follows since $D(T^*)$ is dense in $H$ and
        \begin{align*}
        (I+T^*T)^{-1} T^* &= (I+T^*T)^{-1} T^* (I+TT^*)(I+TT^*)^{-1} \\
        &= (I+T^*T)^{-1} (I+T^*T)T^*(I+TT^*)^{-1} = B \quad\text{on }D(T^*). \qedhere
        \end{align*}
	\end{proof}

    \begin{proof}[Proof of Lemma \ref{lem:D5D7}]
		Without further reference, we use the convergence properties of the cutoff functions $\chi_n(\phi_i)$ throughout this proof and the uniform boundedness of $(\chi_n)_n$, $(\chi_n')_n$, and $(\chi_n'')_n$ whenever applying dominated convergence.
		\begin{enumerate}[label={\roman*)}]
			\item For $f\in C_c^2(\Omega_1)$ define $f_n(x,y)\defeq f(x)\chi_n(\phi_2(y))$ for $n\in\mathbb{N}$ and $(x,y)\in \Omega$. Then $f_n$ lies in $C_c^2(\Omega)=\core$, $f_n \to f$ in $H$ by dominated convergence, and
			\begin{align*}
				Sf_n
				&=  \sum_{i,j=1}^{d_2} \Sigma_{ij} \dyi\dyi[j] f_n +  (\partial_{j} \Sigma_{ij} - \Sigma_{ij} \partial_j\phi_2) \dyi f_n \\
				& \! \begin{multlined}
					= \sum_{i,j=1}^{d_2}  
					\Sigma_{ij} (\chi_n''(\phi_2)\partial_i \phi_2 \partial_j \phi_2 + \chi_n'(\phi_2) \partial_i\partial_j \phi_2) f \\ 
					\qquad\qquad + \partial_{j}\Sigma_{ij} (\chi_n'(\phi_2) \partial_i\phi_2) f
					- \Sigma_{ij} (\chi_n'(\phi_2) \partial_j \phi_2 \partial_i \phi_2) f.
				\end{multlined} \\
			\end{align*}
			If versions of $\partial_{j} \Sigma_{ij}$ are fixed, $Sf_n$ converges to 0 pointwise as $n\to\infty$. Since $f$, $\Sigma_{ij}$ and $\partial_{j}\Sigma_{ij}$ are bounded and $\lvert\nabla\phi_2\rvert^2$ and $\lvert \nabla^2\phi_2\rvert$ are in $L^2(\mu_2)$, see ($\Phi_2$3) and ($\Sigma$2), the convergence is $L^2(\Omega,\mu)$-dominated and thus holds in $H$.
			Since $(S,D(S))$ is closed, we conclude that $f\in D(S)$ and $Sf=0$.
			
			Now, let $g\in L^2(\Omega_1,\mu_1)\supset P(H)$. By Lemma \ref{lem:D_dense} $C_c^2(\Omega_1)$ is dense in $L^2(\mu_1)$, i.e., there is a sequence $(g_n)_{n\in\mathbb{N}}$ in $C_c^2(\Omega_1)\subset D(S)$ converging to $g$ in $L^2(\mu_1)\subset H$. But $Sg_n = 0$ by the above, so $g\in D(S)$ and $Sg=0$ follow since $(S,D(S))$ is closed.
			
			\item For $f\in C_c^2(\Omega_1)$ define $f_n$ as in i) above.
			Then
			\begin{equation*}
				\begin{split}
					Af_n &= ( Q^*\nabla \phi_1)\cdot \grady f_n -( Q \nabla\phi_2)\cdot \gradx f_n \\
					&= f ( Q^*\nabla \phi_1)\cdot (\chi_n'\nabla\phi_2) - \chi_n(\phi_2) ( Q \nabla\phi_2)\cdot \gradx f
					\to - ( Q \nabla\phi_2)\cdot \gradx f
				\end{split}
			\end{equation*}
			pointwise as $n\to\infty$. Note that $\vert \gradx f\vert$ and $\vert f( Q^*\gradx\phi_1)\vert$ are from $C_c(\Omega_1)\subset L^2(\mu_1)$ and that $\vert\nabla \phi_2\vert \in L^2(\mu_2)$, see ($\Phi_2$3). 
			Thus the convergence is $L^2(\mu)$-dominated by Fubini's theorem and holds in $H$.
			
			Since $(A,D(A))$ is closed, we conclude that $f\in D(A)$ and
			\begin{equation}
				\label{eq:preAPs}
				Af = -( Q\nabla\phi_2)\cdot \gradx f.
			\end{equation}
 			Since $P_S(\core) \subset C_c^2(\Omega_1)$, see Lemma \ref{lem:PS:diff}, this proves $P_S(\core)\subset D(A)$ and
 			the formula for $AP_S$ on $\core$. As $P_Sf$ and $Pf$ only differ by a constant for all $f\in H$, it remains to show that $1\in D(A)$ with $A1=0$.

			Since $\chi_n(\phi_1)\in C_c^2(\Omega_1)$ for $n\in\mathbb{N}$, \eqref{eq:preAPs} implies that
			\begin{displaymath}
				A\chi_n(\phi_1) =  -( Q\nabla\phi_2)\cdot (\chi_n'(\phi_1)\gradx\phi_1) \to 0
			\end{displaymath}
			pointwise as $n\to\infty$. With $\vert\nabla\phi_1\vert \in L^2(\mu_1)$ and $\vert\nabla\phi_2\vert\in L^2(\mu_2)$, see ($\Phi_1$3) and ($\Phi_2$3), dominated convergence yields that
			\begin{equation}
				\label{eq:A1=0}
				\chi_n(\phi_1)\to 1 \qquad\text{and}\qquad A\chi_n(\phi_1)\to 0 \qquad\text{in }H\text{ as }n\to\infty.
			\end{equation}
			Since $(A,D(A))$ is closed, $1\in D(A)$ and $A1=0$ follow.
			
			\item For $f\in \core$, we have that $ P_S f\in C_c^2(\Omega_1)$ and $g\defeq APf = -( Q\nabla\phi_2)\cdot \gradx P_S f$ by ii). Thus $g_n\defeq \chi_n(\phi_2)g\in C_c^1(\Omega)$ and by definition of $(A,C_c^1(\Omega))$ we have
			\begin{equation*}
				\begin{split}
					Ag_n
					&= ( Q^* \nabla\phi_1)\cdot (g \chi_n'(\phi_2)\nabla\phi_2 + \chi_n(\phi_2) \grady g) - (Q \grady\phi_2) \cdot (\chi_n(\phi_2)\gradx g) \\
					&\to ( Q^*\nabla\phi_1)\cdot \grady g - ( Q \nabla\phi_2)\cdot\gradx g
				\end{split}
			\end{equation*}
			pointwise as $n\to\infty$. Note that
			\begin{displaymath}
				\grady g = - (\nabla^2\phi_2) Q^*\gradx P_S f
				\qquad\text{and}\qquad
				\gradx g = -(\gradx^2 P_S f) Q \nabla \phi_2.
			\end{displaymath}
			Here, $Ag_n$ is a sum of products $f_1 f_2$ where $f_1$ is a function of $x$, $f_2$ is a function of $y$, and only $f_2$ depends on $n$. For each summand $f_1$ lies in $C_c(\Omega_1)\subset L^2(\mu_1)$ since it involves a derivative of $ P_S f\in C_c^2(\Omega_1)$ and $f_2$ is $L^2(\mu_2)$-dominated uniformly in $n$ since $\lvert\nabla\phi_2\rvert^2 \in L^2(\mu_2)$ and $\lvert \nabla^2\phi_2\rvert \in L^2(\mu_2)$.
			Hence the convergence of $Ag_n$ holds in $H$. Also, $g_n\to g$ in $H$ because $\vert g_n\vert\leq g \in H$.
			
			Since $(A,D(A))$ is closed, $g=APf\in D(A)$ with
			\begin{equation*}
					A^2Pf = Ag = ( Q^*\nabla\phi_1)\cdot \grady g - ( Q \nabla\phi_2)\cdot\gradx g.
			\end{equation*}
			
			\item Let $f\in C_c^2(\Omega_1)$ and define $f_n\in \core$ as in i) above. Then by i) and ii)
			\begin{displaymath}
				f_n\to f \quad\text{and}\quad \cL f_n = Sf_n - Af_n \to - Af \quad\text{in }H\text{ as }n\to\infty.
			\end{displaymath}
			Since $(\cL,D(\cL))$ is closed, it follows that $C_c^2(\Omega_1)\subset D(\cL)$ and $\cL=-A$ on $C_c^2(\Omega_1)$. In particular,
			$\cL\chi_n(\phi_1) = -A\chi_n(\phi_1)$  for all $n$, so that \eqref{eq:A1=0} implies $1\in D(\cL)$ and $\cL 1=0$, again since $(\cL,D(\cL))$ is closed.\qedhere 
		\end{enumerate}
	\end{proof}

    \begin{proof}[Proof of Lemma \ref{lem:phi2_integrals}]
		First, note that all integrals in the claim exist (with finite value) due to the integrability assumption ($\Phi_2$3).
		By assumption we have $\phi_2(y)=\psi(\vert y\vert^2)$ and thus
		\begin{equation}
			\label{eq:psi:diff}
			\partial_i\phi_2 (y) = \psi'(\vert y\vert^2) 2 y_i
		\end{equation}
		for $y\in\Omega_2$. On the $\mu_2$-zero set $\mathbb{R}^{d_2}\setminus \Omega$ we may set $\partial_i\phi_2 (y) \defeq 0$ and $\psi'(\vert y\vert^2)\defeq 0$, so that \eqref{eq:psi:diff} holds for all $y\in \mathbb{R}^{d_2}$. Then
		\begin{displaymath}
			\int_{\Omega_2} \partial_i\phi_2 \,\dmui{2} = \int_{\mathbb{R}^{d_2}} \psi'(\vert y\vert^2)\, 2 y_i \, Z_2^{-1} e^{-\psi(\vert y\vert^2)} \,\dy,
		\end{displaymath}
		which evaluates to zero since the latter integrand is odd. 
		Similarly,
		\begin{equation*}
			\int_{\Omega_2} \partial_i\phi_2 \,\partial_j\phi_2 \,\dmui{2} 
			= \int_{\mathbb{R}^{d_2}} \bigl(\psi'(\vert y\vert^2) 2 y_i\bigr) \, \bigl(\psi'(\vert y\vert^2) 2y_j\bigr) \, Z_2^{-1} e^{-\psi(\vert y\vert^2)} \,\dy,
		\end{equation*}
		and, if $i\neq j$, the latter integrand is odd as a function of $y_i$, so the integral vanishes.
		
		In order to verify $\int_{\Omega_2} (\partial_i\phi_2)^2 \,\dmui{2}=\konst$, we calculate that
		\begin{align*}
			d_2 	\int_{\Omega_2} (\partial_i\phi_2)^2 \,\dmui{2} 
			& = \sum_{k=1}^{d_2} 4 Z_2^{-1} \int_{\mathbb{R}^{d_2}} \psi'(\vert y\vert^2)^2 e^{-\psi(\vert y\vert^2)} \,  y_i^2 \,\dy \\
			& = \sum_{k=1}^{d_2} 4 Z_2^{-1} \int_{\mathbb{R}^{d_2}} \psi'(\vert y\vert^2)^2 e^{-\psi(\vert y\vert^2)} \,  y_k^2 \,\dy 
			= \sum_{k=1}^{d_2} \int_{\Omega_2} (\partial_k\phi_2)^2 \,\dmui{2}.
		\end{align*}
		Here, the second equality follows from the invariance of the Euclidean norm under orthogonal coordinate transformations such as interchanging $y_i$ and $y_k$.

		Finally, recall the cutoff functions from Lemma \ref{lem:cutoffs}: The sequences $(\chi_n(\phi_2))_{n\in\mathbb{N}}$ and $(\chi_n'(\phi_2))_{n\in\mathbb{N}}$ converge to $1$ and $0$, respectively, on $\Omega_2$ dominated by 
		constant functions. Therefore, dominated convergence and integration by parts yield 
		\begin{align*}
				\int_{\Omega_2} \partial_i\partial_j\phi_2 \,\dmui{2} 
				& = \lim_{n\to\infty} \int_{\Omega_2}  \partial_i\partial_j\phi_2 \;\chi_n(\phi_2)\,\dmui{2} \\
				& = \lim_{n\to\infty} - \int_{\Omega_2}  \partial_j\phi_2 \; \bigl(\chi_n'(\phi_2)\partial_i\phi_2 \bigr) \,\dmui{2} + \int_{\Omega_2} \partial_j\phi_2 \, \chi_n(\phi_2) \; \partial_i\phi_2\,\dmui{2} \\
				& = \int_{\Omega_2} \partial_j\phi_2\, \partial_i\phi_2 \,\dmui{2}. \qedhere
		\end{align*}
	\end{proof}
    
    \section{Details on the Coordinate Transformations}
    \label{app:details_coordinate_trafos}
    In this section, we present details on the coordinate transformations employed in proofs of Theorem \ref{thm:main:abstract}, Lemma \ref{lem:essT:suff}, and Lemma \ref{lem:essL:suff}.
    
    \begin{proof}[Proof of Theorem \ref{thm:main:abstract} -- Coordinate Transformation]
        To reduce the general setting to the special case of $\tau=I$ and $b=0$,
        we consider on $\mathbb{R}^{d_1} \times \mathbb{R}^{d_2}$ the coordinate transformation
		\begin{equation*}
			U(x,y)=(U_1(x),U_2(y)) \quad\text{for}\quad U_1(x) = x \quad\text{and}\quad U_2(y) = \tau^{-1}(y+b),
		\end{equation*} 
		and define
		\begin{equation}
			\label{eq:transformation_symmetry_phi2}
			\bar{\phi}_i\defeq \phi_i\circ U_i, \qquad
			\bar{Q}\defeq Q\tau^*, \qquad\text{and}\qquad
			\bar{\Sigma}\defeq (\tau \Sigma \tau^*) \circ U_2.
		\end{equation}
		As detailed below, $\bar{\phi}_1$, $\bar{\phi}_2$, $\bar{\Sigma}$ and $\bar{Q}$ have essentially the same properties as $\phi_1$, $\phi_2$, $\Sigma$ and $Q$, and hence all subsequent definitions apply. All objects in this transformed framework are denoted with a bar on top.
		Note that $\bar{\phi}_1=\phi_1$ and $\bar{\phi}_2 = \psi(\vert \cdot \vert^2)$, i.e., $\bar{\tau}=I$ and $\bar{b}=0$.
		
		The original and transformed framework are equivalent in the following sense: The assumptions \textup{($\Phi_1$1)} - \textup{($\Phi_1$4)}, \textup{($\Phi_2$1)} - \textup{($\Phi_2$5)}, \textup{($\Sigma$1)}, \textup{($\Sigma$2)}, \textup{(eL)} and \textup{(eT)} hold if and only if corresponding assumptions \textup{($\bar{\Phi}_1$1)} - \textup{($\bar{\Phi}_1$4)}, \textup{($\bar{\Phi}_2$1)} - \textup{($\bar{\Phi}_2$5)}, \textup{($\bar{\Sigma}$1)}, \textup{($\bar{\Sigma}$2)}, \textup{(e$\bar{\textup{L}}$)} and \textup{(e$\bar{\textup{T}}$)} hold in the transformed setting (with adapted constants). The hypocoercivity conditions (H) hold if and only if corresponding conditions ($\bar{\textup{H}}$) hold in the transformed setting. In this case the feasible constants $\Lambda_m$, $\Lambda_M$, $N_1$ and $N_2$ in (H) coincide with the constants in ($\bar{\textup{H}}$).
		
		It is straightforward but tedious to verify this assertion. The essential arguments are as follows. 
		First, the image measures $\bar{\mu}_i\circ U_i^{-1}$ and $\bar{\mu}\circ U^{-1}$ can be identified as $\mu_i$ and $\mu$ using the change-of-variables formula. Thus,
        \begin{equation*}
            J_i\colon L^2(\mathbb{R}^{d_i},\mu_i) \to L^2(\mathbb{R}^{d_i},\bar{\mu}_i), \quad f\mapsto f\circ U_i
            \quad\text{and}\quad
            J \colon L^2(\mu) \to L^2(\bar{\mu}), \quad f\mapsto f\circ U
        \end{equation*}
		are unitary isomorphisms. The chain-rule yields $\bar{\core}=C_c^2(\bar{\Omega}) = J(C_c^2(\Omega)) = J(\core)$ and
		\begin{equation*}
			\bar{S}(f\circ U) = (Sf)\circ U \qquad\text{and}\qquad \bar{A} (f\circ U) = (Af) \circ U 
		\end{equation*}
        for all $f\in \core$, i.e., $(\bar{S}, \bar{\core})=(JSJ^{-1}, J(\core))$ and $(\bar{A}, \bar{\core}) = (JSJ^{-1},J(\core))$. Since $J$ is unitary, it follows that
		$(\bar{A}, D(\bar{A})) = (JAJ^{-1},J(D(A)))$ and $(\bar{S}, D(\bar{S})) = (JSJ^{-1},J(D(S)))$.
		
		Using $\bar{\mu}_i\circ U_i^{-1}=\mu_i$ we calculate that $\bar{P}_S = JP_S J^{-1}$ and $\bar{P} = JPJ^{-1}$. By definition of $G$ on $\core$ as $PA^2P$, it follows that $(\bar{G},\bar{\core}) = (JGJ^{-1},J(\core))$. Furthermore,
		$(\bar{A}\bar{P})^* = J(AP)^*J^{-1}$
        and $(I+(\bar{A}\bar{P})^*(\bar{A}\bar{P}))^{-1} = J(I+(AP)^*(AP))^{-1} J^{-1}$ on the respective domains, and therefore $\bar{B} = JBJ^{-1}$ on $D((\bar{A}\bar{P})^*)=J(D((AP)^*))$ and ultimately on $\bar{H}=J(H)$.
		Consequently, we obtain equivalence of the hypocoercivity conditions in both frameworks and that the respective constants coincide.
		
		Equivalence of the assumptions on the original and the transformed potentials follows essentially from the chain rule and the relation $\mu_i=\bar{\mu}_i\circ U_i^{-1}$.
		Moreover, ($\Sigma$1) and ($\Sigma$2) are clearly equivalent to \textup{($\bar{\Sigma}$1)} and \textup{($\bar{\Sigma}$2)}, and there are positive constants $c_1(\tau)$ and $c_2(\tau)$ depending only on  $\tau$ such that
		\begin{equation}
			\label{eq:sigma_bar_vs_sigma}
			c_{\sm{\bar{\Sigma}}} = c_1(\tau) \Cell \qquad\text{and}\qquad M_{\sm{\bar{\Sigma}}} \leq c_2(\tau) \Csigma.
		\end{equation}
		As $(\bar{\sL}, \bar{\core}) = (J\sL J^{-1}, J(\core))$, the image $(I-\bar{\sL})(\bar{\core})$ is dense in $L^2(\bar{\mu})$ if and only if $(I-\sL)(\core)$ is dense in $L^2(\mu)$. Therefore (eL) and \textup{(e$\bar{\textup{L}}$)} are equivalent.
		Since $\bar{\tau}=I$, we have $\bar{Q}\bar{\tau}^* = Q\tau^*$ and thus $(\bar{T},C_c^2(\bar{\Omega}_1)) = (T, C_c^2(\Omega))$. Hence (eT) and (e$\bar{\textup{T}}$) coincide.
		
		Having established the claimed equivalences, the general case is proven as follows. Using the above transformation, we move to a transformed framework where $\bar{\phi}_2$ is radially symmetric, apply the results of Section \ref{subsec:hypocoercivity_conditions}, and obtain that the hypocoercivity conditions (in both the transformed and the original setting) are satisfied. Due to \eqref{eq:sigma_bar_vs_sigma}, the $\Sigma$-dependence of the constants involved is of the form \eqref{eq:sigma_dependence}. Thus the convergence rate is obtained exactly as in the radially symmetric case.
    \end{proof}

    \begin{proof}[Proof of Lemma \ref{lem:essT:suff} -- Coordinate Transformation]
        We show that $(T,C_c^2(\Omega))$ takes the simple form
		\begin{equation}
			\label{eq:T:simple:app}
			T = \Delta - \nabla \phi_1 \cdot \nabla 
		\end{equation}
		under a suitable coordinate transformation. Since $Q$ has full rank and $\tau$ is invertible, $Q\tau^* \tau Q^*\in \mathbb{R}^{d_1\times d_1}$ is symmetric and positive definite. Therefore there is an invertible square matrix $Q_s\in\mathbb{R}^{d_1\times d_1}$ such that $Q\tau^* \tau Q^* = Q_sQ_s^*$. Now, consider the coordinate transformation $U(x)=Q_s x$ on $\mathbb{R}^{d_1}$, set $\bar{\phi}_1 \defeq \phi_1 \circ U$, $\bar{\Omega}_1\defeq \{\bar{\phi}_1 <\infty \}$ and define $\bar{\mu}_1$ accordingly. Then 
		\begin{equation*}
			\iota\colon L^2(\mathbb{R}^{d_1},\mu_1) \to L^2(\mathbb{R}^{d_1},\bar{\mu}_1), \qquad f\mapsto f\circ U
		\end{equation*}
		defines an isometric isomorphism that maps $C_c^2(\Omega_1)$ to $C_c^2(\bar{\Omega}_1)$. The chain rule yields 
		\begin{equation*}
			\bar{T}(\iota(f)) \defeq (\Delta - \nabla\bar{\phi}_1 \cdot \nabla) (\iota(f)) = \iota(Tf)
		\end{equation*}
		for all $f \in C_c^2(\Omega_1)$. 
        It follows that $(\bar{T},C_c^2(\bar{\Omega}_1))$ is densely defined, symmetric and negative semidefinite, too, and that $(T,C_c^2(\Omega_1))$  is essentially $m$-dissipative, i.e., essentially self-adjoint, if and only if $(\bar{T},C_c^2(\bar{\Omega}_1))$ is.
		Furthermore, \textup{($\Phi_1$1)}, \textup{($\Phi_1$5)} and \textup{(eT)$'$} are invariant with respect to the linear coordinate transformation, i.e., $\bar{\phi}_1$, $\bar{\Omega}_1$ and $\bar{\mu}_1$ have corresponding properties. We may therefore assume without loss of generality that $T$ has the simple form \eqref{eq:T:simple:app}.
    \end{proof}

    \begin{proof}[Proof of Lemma \ref{lem:essL:suff} -- Coordinate Transformation]
        We show that $(\sL, \core)$ takes the simple form
		\begin{equation}
			\label{eq:L:simple:app}
			\sL = \Delta_y - \nabla \phi_2 \cdot \grady + \nabla \phi_2 \cdot \gradx - \nabla \phi_1 \cdot \grady 
		\end{equation}
		under a suitable coordinate transformation. Since $\Sigma$ is symmetric and positive definite, there is an invertible $\sigma\in\mathbb{R}^{d\times d}$ such that $\Sigma = \sigma\sigma^*$. Now, on $\mathbb{R}^{2d}$ consider the coordinate transformation 
		\begin{equation*}
			U(x,y)=(U_1(x),U_2(y)) \quad\text{for}\quad U_1(x) = Q(\sigma^{-1})^* x \quad\text{and}\quad U_2(y) = \sigma y.
		\end{equation*} 
		Define $\bar{\phi}_i \defeq \phi_i \circ U_i$, $\bar{\Omega}_i= \{\bar{\phi}<\infty\}$ and $\bar{\mu}_i$ as the probability measure on $\mathbb{R}^d$ obtained by normalizing $e^{-\bar{\phi}_i}\,\dx$. Furthermore, set $\bar{\Omega}=\bar{\Omega}_1 \times \bar{\Omega}_2$ and $\bar{\mu}\defeq \bar{\mu}_1 \times \bar{\mu}_2$.
		
		It is now easily verified that $\mu_i$ coincides with the image measure $\bar{\mu}_i\circ U_i^{-1}$ and that
		\begin{equation*}
			J_i\colon L^2(\mathbb{R}^{d},\mu_i) \to L^2(\mathbb{R}^{d},\bar{\mu}_i), \qquad f\mapsto f\circ U_i
		\end{equation*}
		defines an isometric isomorphism that maps $C_c^2(\Omega_i)$ to $C_c^2(\bar{\Omega}_i)$. Similarly, $J\colon L^2(\mu) \to L^2(\bar{\mu}), f\mapsto f\circ U$ is an isometric isomorphism mapping $\core$ to $C_c^2(\bar{\Omega})$.
		
		Define $(\bar{\sL},C_c^2(\bar{\Omega}))$ by
		\begin{equation*}
			\bar{\sL}f \defeq \Delta_y f - \nabla \bar{\phi}_2 \cdot \grady f + \nabla \bar{\phi}_2 \cdot \gradx f - \nabla \bar{\phi}_1 \cdot \grady f. 
		\end{equation*}
		Then the chain rule yields that
		\begin{equation*}
			\bar{\sL}(Jf) = \bar{\sL}(f\circ U) = (\sL f) \circ U = J(\sL f)
		\end{equation*}
		for all $f\in \core$. So with $(\sL, \core)$ also $(\bar{\sL},C_c^2(\bar{\Omega}))$ is densely defined and dissipative, and
		\begin{equation*}
			(I-\bar{\sL})(C_c^2(\bar{\Omega})) = (I-\bar{\sL})(J(\core)) = J((I-\sL)(\core)).
		\end{equation*}
		Thus $(\sL, \core)$ is essentially $m$-dissipative if and only if $(\bar{\sL},C_c^2(\bar{\Omega}))$ is.
		
		Using the chain-rule and the fact that $\mu_i=\bar{\mu}_i\circ U_i^{-1}$, it is easy to check that our assumptions on $\phi_1$ and $\phi_2$ carry over to $\bar{\phi}_1$ and $\bar{\phi_2}$. Thus we may assume without loss of generality that $\sL$ has the simple form \eqref{eq:L:simple:app}.
    \end{proof}
    
    \section{Details on the Stochastic Representation}
    \label{app:details_stochastic_representation}

    In the following, we give a detailed account of the properties of $(\cL,D(\cL))$ and its semigroup used in the proof of Theorem \ref{thm:associated_process}. We work under the assumptions of Section \ref{sec:stochastic_interpretation}.

       \begin{lemma}
		The semigroup $(T_t)_{t\geq 0}$ on $L^2(\Omega, \mu)$ generated by $(\cL,D(\cL))$ is sub-Markovian, i.e., $0 \leq T_t f \leq 1$ for all $t\geq 0$ and $f\in L^2(\Omega,\mu)$ with $0\leq f\leq 1$, and conservative, i.e., $T_t 1 = 1$ for all $t\geq 0$.
	\end{lemma} 
    
	\begin{proof}
		A straightforward calculation verifies that $(\sL, \core)$ is an abstract diffusion operator on $L^2(\Omega,\mu)$ in the sense of \cite[Appendix B, Definition 1.5]{Eberle_EssSelfAdjointSchroedingerOperators}. Moreover, $\mu$ is invariant for $(\sL, \core)$ in the sense of data condition (D6), i.e., $\int_{\Omega} \sL f \,\dmu = 0$ for all $f\in \core$. Therefore, $(T_t)_{t\geq 0}$ is sub-Markovian by \cite[Appendix B, Lemma 1.9]{Eberle_EssSelfAdjointSchroedingerOperators}.
        Data condition (D7) implies that $\tfrac{\textup{d}}{\textup{d}t}T_t1 = T_t\cL 1 = 0$ for all $t\geq 0$. Hence $T_t1 = T_01 = 1$ for all  $t\geq 0$. 
	\end{proof}
	
	As a generator of a sub-Markovian strongly continuous contraction semigroup, $(\cL,D(\cL))$ defines a sub-Markovian strongly continuous contraction resolvent $(G_\alpha)_{\alpha > 0}$ on $L^2(\Omega,\mu)$ by $G_\alpha \defeq (\alpha - \cL)^{-1}$, see \cite[Propositions I.1.10 and I.4.3]{MR92}. The notions of nests, exceptional sets, and properties holding quasi-everywhere, see e.g. \cite[Section 1]{BBR06}, are in the following understood with respect to this resolvent.

    \begin{lemma}
        \begin{enumerate}[label={\roman*)}]
            \item The compact sets $(F_k)_{k\in\mathbb{N}} \subset \Omega$ where $F_k = \overline{\{\phi_1 < k\}} \times \overline{\{\phi_2 < k\}}$ form a nest. Here, $\overline{\{\phi_i < k\}}$ denotes the closure of the set $\Omega_{i,k}$ as defined in \textup{($\Phi_i$1)}.
            \item The set $\core_\textup{ext}\defeq \{ f+\lambda : f\in \core, \lambda\in \R\}$ is a core of $(\cL,D(\cL))$, forms an algebra and consists of continuous bounded functions.
            \item There is a sequence of continuous functions in $D(\cL)$ separating the points of $\Omega$.
        \end{enumerate}
		
	\end{lemma}
	\begin{proof}
		i) The sets $F_k$ are compact and contained in $\Omega$ as a consequence of ($\Phi_1$1) and ($\Phi_2$1). To conclude that they form a nest, it is sufficient by \cite[Remark III.2.11]{Stannat_GeneralizedDirichletForms} to show that
		\begin{equation*}
			\mathcal{C} \defeq \{f \in D(\cL) : f = 0 \text{ }\mu\text{-a.s. on } \Omega\setminus F_k \text{ for some }k\in\mathbb{N}\}
		\end{equation*}
		is a core of $(\cL,D(\cL))$. But this is immediate since $\core\subset \mathcal{C}$ is a core of $(\cL,D(\cL))$.
		(Note that $\phi_1$ and $\phi_2$ are bounded on every compact subset of $\Omega$, so for all $f\in \core = C_c^2(\Omega)$ we have that $\supp f \subset F_k$ for some $k$, which proves that $\core \subset\mathcal{C}$.)

        ii) Data condition (D7) implies that $\core \subset \core_\textup{ext} \subset D(\cL)$, so $\core_\textup{ext}$ is a core of $(\cL,D(\cL))$. Since $\core = C_c^2(\Omega)$ and $1 \in\core_\textup{ext}$, $\core_\textup{ext}$ is an algebra of continuous bounded functions.

        iii) Since $\Omega$ is open and has a countable dense subset, there is a separating sequence in $\ccinfty(\Omega)\subset D(\cL)$.
	\end{proof}

	\begin{lemma}
		The operator $(\cL,D(\cL))$ is local, i.e., for every continuous $f\in D(\cL)$ we have that $\cL f = 0$ $\mu$-almost everywhere on $\Omega \setminus \supp f$.
	\end{lemma}
	\begin{proof}
		Fix $f$ as above. Since $\core$ is a core of $(\cL,D(\cL))$ we find functions $f_n \in \core$ such that $f_n \to f$ and $\cL f_n \to \cL f$ in $H$ as $n\to \infty$. For every $\varphi \in \ccinfty(\Omega\setminus \supp f)$ the (anti-)symmetry of $S$ and $A$ yields
		\begin{equation*}
			\SPH{\cL f_n}{\varphi} = \SPH{f_n}{S\varphi} + \SPH{f_n}{A\varphi},
		\end{equation*}
		and in the limit
		\begin{equation*}
			\SPH{\cL f}{\varphi} = \SPH{f}{S\varphi + A\varphi} = 0,
		\end{equation*}
		where the last equality is due to the disjoint supports of $f$ and $\varphi$.
		
		Set $g\defeq \cL f \cdot \ind{\Omega\setminus \supp f}$. As in Lemma \ref{lem:D_dense} we obtain that $\ccinfty(\Omega\setminus \supp f)$ is dense in $L^2(\Omega\setminus\supp f, \mu)$, and thus there are functions $\varphi_n \in \ccinfty(\Omega\setminus \supp f)$ such that $\varphi_n \to g$ in $H$ as $n\to \infty$. Consequently,
		\begin{equation*}
			\Vert \cL f \Vert_{L^2(\Omega\setminus\supp f, \mu)}^2 = \SPH{\cL f}{g} = 0. \qedhere
		\end{equation*}
	\end{proof}

    Next, we provide the proof of Theorem \ref{thm:weak_solution_sde}.
    \begin{proof}[Proof of Theorem \ref{thm:weak_solution_sde}]
        First, we fix a sequence of stopping times. Let $\tau_n \defeq \inf \bigl\{ t\geq 0 : (X_t,Y_t) \notin \overline{\Omega_{1,n}} \times \overline{\Omega_{2,n}} \bigr\}$. Then $\tau_n$ is a stopping time since $(X_t,Y_t)_{t\geq 0}$ and $(\mathcal{G}_{t+})_{t\geq 0}$ are right-continuous. By definition $(\tau_n)_{n\in\mathbb{N}}$ is increasing and we have $\tau_n \to\infty$: Each path $\omega \in C([0,\infty);\Omega)$ restricted to a finite interval $[0,T]$ has a compact image, on which $\phi_1$ and $\phi_2$ are bounded, hence $T<\tau_n(\omega)$ for sufficiently large $n$. 

		Next, set $\pi_i\colon \mathbb{R}^{d_1} \to \mathbb{R}, x\mapsto x_i$ for $1\leq i \leq d_1$ and $\rho_j \colon \mathbb{R}^{d_2} \to \mathbb{R}, y\mapsto y_j$ for $1\leq j \leq d_2$ and define the cutoffs $\pi_{i,n} \defeq \chi_n(\phi_1) \chi_n(\phi_2) \pi_i$ and $\rho_{j,n} \defeq \chi_n(\phi_1) \chi_n(\phi_2) \rho_j$ for $n\in\mathbb{N}$. 
        Here, we choose $\chi_n$ in Definition \ref{def:cutoffs} such that $\chi_n = 1$ on $(-\infty, n+\tfrac{1}{2})$, 
        then $\pi_{i,n}=\pi_i$ and $\rho_{j,n}=\rho_j$ on an open superset of $\overline{\Omega_{1,n}} \times \overline{\Omega_{2,n}}$.
        Note that $\pi_{i,n}, \rho_{j,n} \in C_c^2(\Omega)$.
		Thus formula \eqref{eq:intro:generator} yields the following identities on $\overline{\Omega_{1,n}} \times \overline{\Omega_{2,n}}$:
        \begin{align}
        	\label{eq:Lpi_i,n}
        	\cL \pi_{i,n} &= (Q\nabla\phi_2)_i, 
        	&&
        	\cL (\pi_{i,n})^2 = (Q\nabla\phi_2)_i \, 2\pi_{i,n}, \\
        	\label{eq:Lrho_j,n}
        	\cL \rho_{j,n} &= (\divergence \Sigma - \Sigma \nabla\phi_2 - Q^*\nabla\phi_1)_j,
        	&& 
        	\cL (\rho_{j,n} \rho_{k,n}) = 2\Sigma_{jk} + \rho_{k,n} \cL \rho_{j,n} + \rho_{j,n} \cL \rho_{k,n}.
        \end{align}
        
        We inspect the martingales ${M}_t^{[\pi_{i,n}]}$ and ${N}_t^{[\pi_{i,n}]}$ from Theorem \ref{thm:martingale_problem_path_space} and Remark \ref{rem:martingale_solutions:regularity}. Recall that they are defined pointwise, given a continuous/bounded choice of representatives.
        By construction we have $({X}_s,{Y}_s)\in \overline{\Omega_{1,n}} \times \overline{\Omega_{2,n}}$ for $s < \tau_n$.
		Therefore, by \eqref{eq:martingale:m} and \eqref{eq:Lpi_i,n},
		\begin{equation}
        \label{eq:M:localization:line1}
			{M}_{t\wedge \tau_n}^{[\pi_{i,n}]} = {X}_{t\wedge \tau_n}^i - {X}_0^i - \int_{0}^{t\wedge \tau_n} (Q\nabla\phi_2)_i({X}_s,{Y}_s) \,\textup{d}s \quad\text{for all }t\geq 0.
		\end{equation}
		As $\cL (\pi_{i,n}^2)-2\pi_{i,n}\cL \pi_{i,n}=0$ on $\overline{\Omega_{1,n}} \times \overline{\Omega_{2,n}}$ by \eqref{eq:Lpi_i,n}, we obtain from \eqref{eq:martingale:n} that
		\begin{equation}
        \label{eq:N:localization:line1}
			{N}_{t\wedge \tau_n}^{[\pi_{i,n}]} = \Bigl({M}_{t\wedge \tau_n}^{[\pi_{i,n}]}\Bigr)^2 \quad\text{for all }t\geq 0.
		\end{equation}
		Since ${M}^{[\pi_{i,n}]}$ and ${N}^{[\pi_{i,n}]}$ are continuous martingales, so are the stopped processes $({M}_{t\wedge {\tau_n}}^{[\pi_{i,n}]})_{t\geq 0}$ and $({N}_{t\wedge {\tau_n}}^{[\pi_{i,n}]})_{t\geq 0}$.
        It follows using \eqref{eq:M:localization:line1} and \eqref{eq:N:localization:line1} that the processes ${M}^{[\pi_i]}$ and ${N}^{[\pi_i]}$ defined by 
        \begin{equation*}
            {M}_t^{[\pi_{i}]} \defeq {X}_{t}^{i} - {X}_{0}^{i} - \int_{0}^{t} (Q\nabla\phi_2)_i({X}_s,{Y}_s) \,\textup{d}s \quad\text{and}\quad
            {N}_t^{[\pi_i]} \defeq ({M}_t^{[\pi_i]})^2
        \end{equation*}
        are continuous local martingales. In particular, ${M}^{[\pi_i]}$ has quadratic variation zero. Hence ${M}_t^{[\pi_{i}]} = {M}_0^{[\pi_{i}]}= 0$ for all $t\geq 0$ 
        $\mathbb{P}_{h\mu}$-almost surely. By definition of ${M}^{[\pi_i]}$ this proves \eqref{eq:SDE:X}.
		
		Next, we turn to \eqref{eq:SDE:Y}. 
        Using \eqref{eq:martingale:m} and \eqref{eq:Lrho_j,n}, we obtain that
        \begin{equation*}
            \label{eq:M:localization:line2}
			{M}_{t\wedge \tau_n}^{[\rho_{j,n}]} = {Y}_{t\wedge \tau_n}^j - {Y}_0^j - \int_{0}^{t\wedge \tau_n} (\divergence \Sigma - \Sigma \nabla\phi_2 - Q^*\nabla\phi_1)_j({X}_s,{Y}_s) \,\textup{d}s
		\end{equation*}
        and conclude as above that ${M}^{[\rho_j]}$ defined by
		\begin{equation}
			\label{eq:SDE:Y_tilde}
			{M}_t^{[\rho_{j}]} \defeq {Y}_{t}^{j} - {Y}_{0}^{j} - \int_{0}^{t} (\divergence \Sigma - \Sigma \nabla\phi_2 - Q^*\nabla\phi_1)_j({X}_s,{Y}_s) \,\textup{d}s
		\end{equation}
        is a continuous local martingale.
        
		We compute the quadratic covariation $\left[{M}^{[\rho_j]}, {M}^{[\rho_k]}\right]$. 
        Set ${M}^{[\rho_j\pm \rho_k]} \defeq {M}^{[\rho_j]} \pm {M}^{[\rho_k]}$. Then
		\begin{equation*}
			{M}_{t\wedge \tau_n}^{[\rho_j\pm \rho_k]} = {M}_{t\wedge \tau_n}^{[\rho_{j}]} \pm {M}_{t\wedge \tau_n}^{[\rho_{k}]} 
			= {M}_{t\wedge \tau_n}^{[\rho_{j,n}]} \pm {M}_{t\wedge \tau_n}^{[\rho_{k,n}]}
			= {M}_{t\wedge \tau_n}^{[\rho_{j,n} \pm \rho_{k,n}]},
		\end{equation*}
        where the last equality is due to linearity of \eqref{eq:martingale:m}. Using \eqref{eq:martingale:n} and \eqref{eq:Lrho_j,n}, we compute that
		\begin{equation*}
			{N}_{t\wedge \tau_n}^{[\rho_{j,n}\pm \rho_{k,n}]} = \Bigl({M}_{t\wedge \tau_n}^{[\rho_{j,n}\pm \rho_{k,n}]}\Bigr)^2 - \int_{0}^{t\wedge \tau_n}  2\Sigma_{jj}({Y}_s) + 2\Sigma_{kk}({Y}_s) \pm 4\Sigma_{jk}({Y}_s) \,\textup{d}s.
		\end{equation*}
		We conclude as above that $N^{[\rho_j\pm \rho_k]}$ defined via
		\begin{equation}
			\label{eq:covariation}
			{N}_{t}^{[\rho_{j}\pm \rho_{k}]} \defeq \Bigl({M}_{t}^{[\rho_{j}\pm \rho_{k}]}\Bigr)^2 - \int_{0}^{t}  2\Sigma_{jj}({Y}_s) + 2\Sigma_{kk}({Y}_s) \pm 4\Sigma_{jk}({Y}_s) \,\textup{d}s
		\end{equation}
		is a continuous local martingale. Denoting the $i$-th standard unit vector by $e_i$, we have 
        \begin{equation*}
            2\Sigma_{jj} + 2\Sigma_{kk} \pm 4\Sigma_{jk}= 2 (e_j \pm e_k)^T \Sigma (e_j \pm e_k) = 2 \vert \sigma^* (e_j \pm e_k) \vert^2 \geq 0,
        \end{equation*}
        so the integral term in \eqref{eq:covariation} defines an increasing, continuous, adapted process and can be identified as the quadratic variation of ${M}^{[\rho_j\pm \rho_k]} = {M}^{[\rho_j]} \pm {M}^{[\rho_k]}$. It follows that
		\begin{equation*}
			\left[{M}^{[\rho_j]},{M}^{[\rho_k]}\right]_t =  \int_{0}^{t} 2 \Sigma_{jk}({Y}_s) \,\textup{d}s.
		\end{equation*}
		
		By assumption $\Sigma$ is bounded and uniformly strictly elliptic. Therefore $\sigma$ is invertible, and both $\sigma$ and $\sigma^{-1}$ are continuous and bounded. Thus  $B_t^{i} \defeq \sum_{j=1}^{d_2} \int_{0}^{t} \frac{1}{\sqrt{2}} (\sigma^{-1})_{ij}({Y}_s) \,\textup{d}{M}_s^{[\rho_j]}$ is a continuous local martingale with
		\begin{equation}
			\begin{aligned}
				\left[B^{i}, B^{j} \right]_t 
				&= \sum_{k,l=1}^{d_2} \frac{1}{2} \int_{0}^{t} (\sigma^{-1})_{ik}({Y}_s) (\sigma^{-1})_{jl}({Y}_s) \,\textup{d}[{M}^{[\rho_k]},{M}^{[\rho_l]}]_s \\
				&= \sum_{k,l=1}^{d_2} \int_{0}^{t} (\sigma^{-1})_{ik}({Y}_s) \Sigma_{kl}({Y}_s) (\sigma^{-1})_{jl}(Y_s) \,\textup{d}s = \delta_{ij} t.
			\end{aligned}
		\end{equation}
		By L\'evy's characterization of Brownian motion, $(B_t)_{t\geq 0}$ is a $d_2$-dimensional standard Brownian motion.
		Finally, $(\int_{0}^{t}\sqrt{2} \sigma({Y}_s) \,\textup{d}B_s)_j = {M}_t^{[\rho_j]}$ by construction, which can be inserted into  \eqref{eq:SDE:Y_tilde}. This proves \eqref{eq:SDE:Y}. 
	\end{proof}
\end{document}